\documentclass[a4paper,11pt,reqno]{amsart}

\usepackage{relsize}
\usepackage{cite}
\usepackage{hyperref}
\usepackage{amsmath,amsthm}

\usepackage{float}

\usepackage{amssymb}
\usepackage{lipsum,enumitem}

\usepackage{mathtools}

\usepackage[top=2.3cm, bottom=2.3cm, left=2.2cm, right=2.2cm]{geometry}
\usepackage{calrsfs}

\usepackage[utf8]{inputenc}

\theoremstyle{definition}
\newtheorem{theorem}{\textsc{\textbf{Theorem}}}[section]
\newtheorem{corollary}[theorem]{\textsc{\textbf{Corollary}}}
\newtheorem{proposition}[theorem]{\textsc{\textbf{Proposition}}}
\newtheorem{lemma}[theorem]{\textsc{\textbf{Lemma}}}

\newtheorem*{theoremNN}{Theorem}
\newtheorem*{corollaryNN}{Corollary}

\theoremstyle{definition}
\newtheorem{definition}[theorem]{\textsc{\textbf{Definition}}}
\newtheorem{example}[theorem]{\textsc{\textbf{Example}}}
\newtheorem{notation}[theorem]{\textsc{\textbf{Notation}}}
\newtheorem{remark}[theorem]{\textsc{\textbf{Remark}}}
\newtheorem{conjecture}[theorem]{\textsc{\textbf{Conjecture}}}

\newcommand{\numberset}{\mathbb}
\newcommand{\N}{\numberset{N}}
\newcommand{\Z}{\numberset{Z}}
\newcommand{\Q}{\numberset{Q}}
\newcommand{\R}{\numberset{R}}

\newcommand{\F}{\numberset{F}}

\newcommand{\mP}{\mathcal{P}}

\newcommand{\mL}{\mathcal{L}}

\newcommand{\mB}{\mathcal{B}}

\newcommand{\mM}{\mathcal{M}}

\newcommand{\dH}{d_\textnormal{H}}
\newcommand{\sH}{\sigma_\textnormal{H}}

\newcommand{\mH}{\mathcal{H}}
\newcommand{\at}{\textnormal{At}}

\newcommand{\Pro}{\numberset{P}}

\newcommand{\pe}{\nu}
\newcommand{\mee}{\vartriangleright}

\newcommand{\EE}{{E}}
\newcommand{\OO}{{O}}

\newcommand{\oo}{\"o}
\newcommand\qbin[3]{\left[\begin{matrix} #1 \\ #2 \end{matrix} \right]_{#3}}

\newcommand{\EL}{\textnormal{ESF}}

\newcommand{\HH}{\textnormal{H}}
\newcommand{\rk}{\textnormal{rk}}

\newcommand{\crit}{\textnormal{crit}}
\newcommand{\inn}{\textnormal{in}}
\newcommand{\lead}{\textnormal{le}}

\newcommand{\wH}{\omega_\textnormal{H}}
\newcommand{\defi}{\triangleq}

\newcommand{\meet}{\wedge}
\newcommand{\join}{\vee}

\numberwithin{equation}{section}
\allowdisplaybreaks

\setlist[enumerate]{topsep=3pt}
\setlist[itemize]{topsep=3pt}

\title[Whitney Numbers of Combin. Geometries and Higher-Weight Dowling Lattices]{Whitney Numbers of Combinatorial Geometries \\ and Higher-Weight Dowling Lattices}
\author[Alberto Ravagnani]{Alberto Ravagnani$^*$}
\address{School of Mathematics and Statistics, University College Dublin, Belfield, Ireland}
\curraddr{}
\email{alberto.ravagnani@ucd.ie}
\thanks{$^*$The author was partially supported by the Swiss National Science Foundation through grant n. P2NEP2\_168527 and by the the Marie Curie Research Grants
Scheme, grant n. 740880.}



\setlist[itemize]{leftmargin=\parindent}
\setlist[enumerate]{leftmargin=\parindent}

\usepackage{setspace}

\begin{document}

\maketitle
\thispagestyle{empty}

\begin{abstract}
We study the Whitney numbers of the first kind of combinatorial geometries. 
The first part of the paper is devoted to general results relating the M\oo bius functions of nested atomistic lattices,
extending some classical theorems in combinatorics. We then specialize our results to restriction geometries, i.e., 
to sublattices $\mL(A)$ of the lattice of subspaces of an $\F_q$-linear space, say $X$, generated by a set of projective points $A \subseteq X$.
In this context, we introduce the notion of \textit{subspace distribution}, and show that 
partial knowledge of the latter is equivalent to partial knowledge of the Whitney numbers of $\mL(A)$.
This refines a classical result by Dowling.

The most interesting applications of our results are to be seen in the theory of \textit{higher-weight Dowling lattices} (HWDLs), to which we dovote the 
second and most substantive part of the paper.
These combinatorial geometries were introduced by Dowling in 1971 in connection with fundamental problems in coding theory, and further studied, among others,
 by Zaslavsky, Bonin, Kung, Brini, and Games. To date, still very little is known about these lattices. In particular, the techniques to 
 compute their Whitney numbers have not been discovered yet. In this paper, we bring forward the 
theory of HWDLs, computing their Whitney numbers for new infinite families of parameters.
Moreover, we show that the second Whitney numbers of HWDLs are polynomials in the underlying field size $q$, whose coefficients are 
expressions involving the Bernoulli numbers. This reveals a new link between combinatorics, coding theory, and number theory.
We also study the asymptotics of the Whitney numbers of HWDLs as the field size grows, giving upper bounds and exact estimates in some cases. 
In passing, we obtain new results on the density functions of error-correcting codes.
\end{abstract}


\bigskip

\section*{Introduction and Motivation}

The characteristic polynomial 
 is one of the most important combinatorial invariants of a finite graded lattice. 
Its coefficients are integers called \textit{Whitney numbers} (\textit{of the first kind}), and capture properties of the M\oo bius functions of the lattice elements having a fixed rank. 

The information encoded by characteristic polynomials of lattices is particularly relevant to problems in enumerative and extremal combinatorics. Important applications can be seen, for example, to the theory of hyperplane arrangements, to graph colouring questions, and to the theory of Critical Problems introduced by Crapo and Rota; see~\cite{crapo1970foundations,kung1996critical} among many others.

Giving explicit expressions for the characteristic polynomials of families of lattices, and therefore for their Whitney numbers, is a central problem in combinatorics. Several techniques have been developed since Rota's seminal paper on M\oo bius functions~\cite{rotaGC}. Some of these techniques, including Stanley's Modular Factorization Theorem and its generalizations, rely on the fact that certain characteristic polynomials split into linear factors in $\Z[\lambda]$, and that such a (unique) factorization can be given a precise combinatorial interpretation \cite{stanley1971modular,stanley1972supersolvable,
terao1981generalized,zaslavsky1982signed,
blass1997mobius,blass1998characteristic,sagan1999characteristic,hallam2015factoring}. Other methods reduce the computation of characteristic polynomials to point-counting problems over finite fields. These find applications in the theory of real hyperplane and subspace arrangements; see for example~\cite{blass1998characteristic,athanasiadis1996characteristic,bjorner1997subspace}. 
Unfortunately, not all characteristic polynomials of lattices can be computed by applying one of the above mentioned general methods.

%

In this paper, we focus on the Whitney numbers of certain 
\textit{combinatorial geometries} and, more generally, of
\textit{atomistic lattices}. 
  Given an $\F_q$-linear space $X$ and a set of vectors $A \subseteq X$, we consider the geometric lattice $\mL(A)$ whose elements are those subspaces of $X$ spanned by some elements of~$A$. Such a lattice is sometimes called a \textit{restriction geometry} of the lattice of subspaces of $X$.

We start by establishing some simple ``decomposition'' formul{\ae}
for the M\oo bius function of an arbitrary atomistic lattice, extending some classical combinatorial results by Whitney, Stanley, and Crapo. These are then applied to restriction  geometries. More in detail, inspired by the coding theory literature, we study the Whitney numbers of a lattice of the form~$\mL(A)$ in connection with the concept of \textit{subspace distribution}. The latter counts the number of subspaces of $X$ of given dimension $k$ that do not contain any vector from $A$. In this context, we show that \textit{partial} information on subspace distributions yields \textit{partial} information on characteristic polynomials, and vice versa.
This refines a theorem of Dowling \cite{dowling1971codes}, and allows one to explicitly compute the Whitney numbers of certain geometries whose characteristic polynomials do not split into linear factors.


Interesting applications of these results connecting Whitney numbers and subspace distributions can be found in the theory of \textit{higher-weight Dowling lattices} (HWDLs), to which we devote most of the paper. These geometries were introduced by Dowling in 1971, and have so far resisted any attempt to compute their Whitney numbers or characteristic polynomials, with only few exceptions~\cite{zaslavsky1987mobius,bonin1993modular,bonin1993automorphism,
kung1996critical, brini1982some, games1983packing}. 
Recall that
the higher-weight Dowling lattice $\mH(q,n,d)$ is the subgeometry of $\F_q^n$ generated by those vectors whose \textit{Hamming weight} is upper-bounded by $d$; see Definition \ref{def:HWDL}.
By generalizing one of these geometries (namely, the $q$-analogue of the \textit{partition lattice}), Dowling introduced in \cite{dowling1973class} the class of \textit{Dowling lattices} based on finite groups. Higher-weight Dowling lattices are closely
 connected to a fundamental unsolved problem intersecting coding theory, extremal combinatorics, and finite geometry, known as the \textit{MDS Conjecture}. The latter was proposed Segre in the 50's \cite{segre1955curve}. Being able to compute the Whitney numbers of $\mH(q,n,d)$ for all parameter sets would lead to solving this important open problem.

At the time of writing this article, the techniques for studying higher-weight Dowling lattices have not been discovered yet. 
As Zaslavsky observes in \cite[Section 7.6]{zaslavsky1987mobius}, ``this is one of the important open problems in matroid theory''. 

In this paper, we bring forward the theory of higher-weight Dowling lattices, studying properties of these of various flavours (duality, explicit formul\ae, asymptotics, polynomiality). In particular, we compute their Whitney numbers for some new infinite parameter sets. One of our main results is the following formula for the second Whitney number of HWDLs.

%
%

 \begin{theoremNN} 
 For all integers $n \ge d \ge 2$ and all prime powers $q$, the second Whitney number of $\mH(q,n,d)$ is given by
 \begin{align*} w_2(q,n,d) &= (q^{n-1}-1) \, \sum_{j=1}^d \binom{n}{j} (q-1)^{j-2} \ - \sum_{1 \le \ell_1<\ell_2 \le n} \Biggl[ q^{n-\ell_1-1}\left( \sum_{j=0}^{d-1} \binom{n-\ell_2}{j} (q-1)^j\right)  \\ 
&+ \sum_{j=d}^{n-\ell_2}\sum_{h=0}^{d-1} \binom{n-\ell_2}{j}   \binom{n-\ell_1-1}{h}(q-1)^{j+h} \\ 
&+ \sum_{s=d}^{n-\ell_2} \ \sum_{t=0}^{d-2} \binom{n-\ell_2}{s} \binom{n-\ell_1-1-s}{t} (q-1)^{s+t} \sum_{\nu=d-t}^s \gamma_q(s,s-d+t+2,\nu) \Biggr],
 \end{align*} 
 where the $\gamma_a(b,c,\nu)$'s are the \textit{agreement numbers} introduced in  Definition \ref{def:cn}. 
 \end{theoremNN}
 
The agreement numbers are integers defined recursively, which can be given a precise combinatorial interpretation. It turns out that $\gamma_a(b,c,\nu)$ is a polynomial in the variable $a$ (for any fixed $b$, $c$ and $\nu$) whose coefficients are are expressions
in the Bernoulli numbers; see Remark~\ref{rem:ber} for a more precise statement. This fact connects combinatorics, coding theory and number theory. Furthermore, it implies the following polynomiality result.

 \begin{corollaryNN} 
 For all $n \ge d \ge 2$, the second Whitney number of $\mH(q,n,d)$ is a polynomial in $q$.
 \end{corollaryNN}

We include a more detailed description of the contributions made by this paper. This also serves to illustrate how the article is organized.

\begin{itemize} \setlength\itemsep{0.1em}
\item After a short introduction to posets and combinatorial geometries, in Section \ref{sec:atomistic} 
we establish various formul{\ae} that relate the M\oo bius function of nested atomistic lattices.

\item In Section \ref{sec:wnsd} we apply some of the results of Section~\ref{sec:atomistic} to restriction geometries of the form $\mL(A)$. This allows us to refine a result by Dowling, and to relate Whitney numbers and subspace distributions via invertible transformations. We also include examples of how these formul{\ae} can be used to compute the Whitney numbers  of certain non-supersolvable lattices.

\item In Sections \ref{sec:dms} we study combinatorial geometries of the form $\mL(A_1 \cup \cdots \cup A_L)$, where the $A_\ell$'s are linear spaces. The main result is a (partial) duality theorem between the Whitney numbers of $\mL(A_1 \cup \cdots \cup A_L)$ and those of $\mL(A_1^\perp \cup \cdots \cup A_L^\perp)$, where $A_\ell^\perp$ is the orthogonal 
of~$A_\ell$ with respect to a symmetric non-degenerate bilinear.

\item In Sections \ref{sec:HWDL1}--\ref{sec:HWDL6} we apply and extend the results of the first part of the paper to study HWDLs. Section \ref{sec:HWDL1} contains the main definitions and briefly surveys the known results. 

\item Section \ref{sec:HWDL2} contains some new general properties of HWDLs, including ``reduction formul{\ae}'' for their Whitney numbers and duality results.

\item Sections \ref{sec:HWDL3} and \ref{sec:HWDL4} are devoted to computational results. Applying various techniques, in Section~\ref{sec:HWDL3} we give closed formul{\ae} for the Whitney numbers of higher-weight Dowling lattices for some new infinite parameter sets. In Section \ref{sec:HWDL3} we instead define the \textit{agreement numbers} via a recursion, and express the second Whitney number of $\mH(q,n,d)$ in terms of these.


\item 
Although giving formul{\ae} for the Whitney numbers of these lattices seems to be very difficult in general, quite detailed information can be obtained on their \textit{asymptotic behaviour} as the field size grows. We study this problem in
Sections \ref{sec:HWDL5a} and \ref{sec:HWDL5}. We start by establishing some 
results on the \textit{density} of error-correcting codes (Section~\ref{sec:HWDL5}).
In Section \ref{sec:HWDL5a}, we apply these to determine the asymptotics of the Whitney numbers of higher-weight Dowling lattices.

\item We conclude the paper with some results on the polynomiality in $q$ of the Whitney numbers of higher-weight Dowling lattices and their intriguing connection with Bernoulli numbers.

\end{itemize}

\bigskip

\section{Preliminaries} \label{sec:prel}


\subsection{Posets and Lattices}

We briefly recall some definitions from poset theory, and establish the notation for the remainder of the paper. We refer the reader to \cite[Chapter 3]{stanleyec} for further details.

A \textbf{poset} is a pair $(\mP,\le)$ where $\mP$ is a non-empty set and $\le$ is an order relation on $\mP$. In the sequel we abuse notation and denote by $\mP$ the pair $(\mP,\le)$, when the order is clear from context. 
We say that $\mP$ is \textbf{trivial} if it has cardinality $|\mP|=1$.

If $\mP$ has a minimum element, say $0_\mL$, then an \textbf{atom} of $\mP$ is an element $a \in \mL$ with $0_\mL \lessdot a$, where $\lessdot$ denotes the covering relation \cite[p. 279]{stanleyec}.
The set of atoms of $\mP$ is $\at(\mP)$.

We denote by $\mu_\mP$ the 
M\"obius function of a poset $\mP$. If $\mP$ has a minimum element $0_\mL$, then we simply write
$\mu_\mP(x)$ for $\mu_\mP(0_\mL,x)$. 

In this paper we focus on lattices. A \textbf{lattice} is a poset $(\mL, \le)$ where every elements $x,y \in \mL$ have a unique meet, $x \wedge y$, and a unique join, $x \vee y$.
The rank function of a graded lattice $\mL$ is denoted by
$\rk_\mL: \mL \to \N$. The rank of $\mL$ is $\rk(\mL)$.
Every finite lattice $\mL$ has a minimum and a maximum element, denoted by $0_\mL$ and $1_\mL$ respectively (or simply by $0$ and $1$ if no confusion arises). If $\mL$ is graded, these are the only elements of rank $0$ and $\rk(\mL)$, respectively.  Observe moreover that for any finite lattice $\mL$ we have $\at(\mL) \neq \emptyset$ whenever $\mL$ is non-trivial. 

Very desirable lattice properties are semi-modularity and modularity. 
A graded lattice $\mL$ is called \textbf{semi-modular} if its rank function satisfies the  inequality
\begin{equation} \label{semimod}
\rk_\mL(x \meet y) + \rk_\mL(x \join y) \le \rk_\mL(x) + \rk_\mL(y) \quad \mbox{for all } x,y \in \mL.
\end{equation}
An element $x \in \mL$ is \textbf{modular} if equality holds in   
(\ref{semimod}) for all $y \in \mL$. The lattice $\mL$ is \textbf{modular}
if all its elements are modular.
A lattice $\mL$ is \textbf{atomistic} if every non-zero element $x \in \mL$ is the join of a set of atoms of $\mL$, and it
is \textbf{geometric} when it is finite, graded, semi-modular, and atomistic.

We also recall some fundamental combinatorial invariants of a  lattice.

\begin{definition}
Let $\mL$ be a finite graded lattice, and let $i \in \N$. The $i$-th \textbf{Whitney number} (\textbf{of the first kind}) of~$\mL$ is the integer
$$w_i(\mL) \defi\sum_{\substack{x \in \mL \\ \rk_\mL(x)=i}} \mu_{\mL}(x),$$
where the sum over an empty index set is $0 \in \Z$ by definition. The \textbf{characteristic polynomial} of~$\mL$, denoted by
$\chi(\mL;\lambda) \in \Z[\lambda]$, is
defined as
$$\chi(\mL;\lambda) \defi \sum_{x \in \mL} \mu_\mL(x) \  \lambda^{\rk(\mL)-
\rk_\mL(x)} = \sum_{i \in \N} w_i(\mL) \ \lambda^{\rk(\mL)-i}.$$
\end{definition}

The following notation will be used throughout the paper to make statements and proofs more compact.

\begin{notation} \label{notaz1}
If $\mL$ is a lattice and 
 $A \subseteq \at(\mL)$ is a (possibly empty) subset, we denote by $\mL(A)$ the sublattice of $\mL$ made of
those elements that can be written as the join of some elements of $A$, together with $0 \in \mL$.
Clearly, $\mL(A)$ is an atomistic lattice with respect to the order induced by~$\mL$, whenever $A \neq \emptyset$.
We denote the M\"{o}bius function of $\mL(A)$ by $\mu_A$, if~$\mL$ is clear from context. Moreover, 
for $x \in \mL$, we denote the maximum element $y \in \mL(A)$ such that $y \le x$ by 
$$x^A \defi \bigvee \left( \left\{ a \in A \mid a \le x \right\} \cup \{0\} \right).$$
\end{notation}

A standard example of a geometric lattice is the set of subspaces of a finite-dimensional vector space $X$ over a finite field. 
 Its combinatorial invariants can be conveniently expressed in terms of $q$-ary binomial coefficients. 


\begin{example} \label{ex:primo}
Let $X$ be a non-zero vector space of finite dimension over a finite field $\F_q$. Let $n$ denote its dimension. The set $\mL$ of subspaces of $X$, ordered by inclusion, is a geometric lattice. The atoms of $\mL$ are the 1-dimensional subspaces of $X$. The meet of 
$U,V \le X$ is $U \cap V$, and their join is $U+V$.
It is easy to see that
$\mL$ is modular, and that its rank is the dimension of $X$ over $\F_q$. More generally, the rank of a subspace $V \le X$ its $\F_q$-dimension. The characteristic polynomial and the Whitney numbers of $\mL$ are given by
$$\chi(\mL;\lambda) \ = \ \prod_{i=0}^n (\lambda-q^i), \qquad w_i(\mL)=(-1)^iq^{\binom{i}{2}}\qbin{n}{i}{q} \mbox{ for $i \in \N$}.$$
 \end{example}
The formul{\ae} in Example \ref{ex:primo} are very well-known. In the sequel, we will use these without referring to them explicitly.

We conclude this section by mentioning some properties of $q$-ary coefficients will be used repeatedly throughout the paper.

\begin{lemma} \label{proprq}
For all $r,s,t \in \Z$ we have
$$\qbin{r}{s}{q}=\qbin{r}{r-s}{q}, \qquad \qbin{r}{s}{q} \qbin{s}{t}{q}=\qbin{r}{t}{q} \qbin{r-t}{r-s}{q}.$$
\end{lemma}

More properties of $q$-ary coefficients can be found in various combinatorics textbooks; see~\cite{stanleyec} among many others.

\subsection{Combinatorial Geometries and Their Whitney Numbers} \label{sec:CP}

In this paper we are interested in the Whitney numbers of combinatorial geometries that are sublattices of the lattice of subspaces of a linear space over a finite field. In the sequel, $q$ is a prime power and $\F_q$ is the finite field with $q$ elements. We denote by $\rk(V)$ the $\F_q$-dimension of a linear space $V$.

\begin{notation} \label{notaz X}
Throughout this section, $X$ is a non-zero finite-dimensional vector space $X$ over~$\F_q$. The dimension of $X$ is $n$. We let $\mL$ be the lattice of subspaces of~$X$ ordered by the inclusion $\le$.
For a subset $A \subseteq X$, let $\Pro(A)$ be the set (possibly empty)
of 1-dimensional subspaces of~$X$ spanned by some element of $A$. 
Then $\mL(A) \defi \mL(\Pro(A))$ is the lattice of subspaces of~$X$ that have a basis made of elements of $A$, together with the zero space
(see Notation~\ref{notaz1}). The M\"obius function of $\mL(A)$ is denoted by $\mu_A$.
We also let
$V^A \defi V^{\Pro(A)}$ for all $V \le X$, and denote by $\langle A \rangle$ the smallest subspace of~$X$ containing all the elements of $A$. Its dimension is 
$\rk(A)$. We call a lattice of the form $\mL(A)$ the \textbf{combinatorial} \textbf{geometry} over the set $A$.
\end{notation}

In this work, we study the Whitney numbers of combinatorial geometries in connection to the concept of subspace distribution.

\begin{definition} \label{defavoids} \label{def:alpha}
Let $X$ be a finite-dimensional space over $\F_q$, and let 
$A \subseteq X$ be a subset. 
A space $V \le X$ is said to \textbf{distinguish} (or \textbf{avoid}) $A$ if
either $V \cap A = \emptyset$, or $V \cap A = \{0\}$. If this is not the case, then we say that $V$ \textbf{meets} $A$ and write $V \mee A$.
 For $k \in \N$, we let
 $$\alpha_k(X,A) \defi \#\{V \le X \mid \rk(V)=k, \ V \not\mee A\}.$$ The sequence 
$(\alpha_k(X,A) \mid k \in \N)$ is the \textbf{subspace distribution} associated with $X$ and $A$. More generally, for $U \le X$ we let
$\alpha_k(U,A) \defi \#\{V \le U \mid \rk(V)=k, \ U \not\mee A\}$.
\end{definition}

\begin{remark} \label{rem:imp}
Computing $\alpha_k(X,A)$ for an arbitrary $A \subseteq X$ is difficult in general. This is connected to a fundamental question in extremal combinatorics, known as the \textbf{Critical Problem}. The latter was proposed by Crapo and Rota in 1970
and asks to compute the \textbf{critical exponent} of $\mL(A)$, defined by
$$\crit(\mL(A)) \defi \min \{r \in \N \mid \chi(\mL(A);q^r) \neq 0\}.$$
It follows from~\cite[Chapter 16]{crapo1970foundations} that
$$\crit(\mL(A)) =  n-\max\{k \in \N \mid \alpha_k(X,A) \neq 0\}  =  \rk(A)-\max\{k \in \N \mid \alpha_k(\langle A \rangle,A) \neq 0\}.$$
\end{remark}

The Critical Problem 
admits several generalizations to matroids, polymatroids, relations, and to other combinatorial structures; see \cite{kung1996critical} and the references therein for a detailed overview. Various foundamental questions in combinatorics, such as coloring problems, can be formulated as an instance of the Critical Problem.

In Section \ref{sec:wnsd} we will return to the connections between Whitney numbers of combinatorial geometries and subspace distributions.

\section{The M\"{o}bius Function of an Atomistic Lattice}
\label{sec:atomistic}

In this section we prove some formul{\ae} that relate the M\"{o}bius function of an atomistic lattice~$\mL$ to the M\"{o}bius function of certain sublattices $\mL' \subseteq \mL$. These will be applied in later sections. As these results seem  interesting in their own right and we could not find them in any reference, we state them for arbitrary finite atomistic lattices and group them together in an independent section.

The formul{\ae}  that we present extend some classical theorems in combinatorics, such as {Whitney Theorem}, Stanley's {Modular Factorization Theorem} for geometric lattices, and Crapo's formula for nested geometries.

\begin{notation} 
In the sequel, we work with a fixed non-trivial finite and atomistic lattice $\mL$. We do \textit{not} assume that $\mL$ is graded or geometric, unless explicitly stated. We follow the notation of Section~\ref{sec:prel}. Recall in particular Notation \ref{notaz1}.
\end{notation}

We start with a simple lemma, that will be used repeatedly throughout the paper. 

\begin{lemma} \label{lem:toA}
Let $A \subseteq \at(\mL)$ be a set, $x \in \mL(A)$, and $y \in \mL$.  Then $x \le y$ if and only if $x \le y^A$.
\end{lemma}

The first result that we present is a formula that relates the M\"obius functions of  two nested lattices. It generalizes~\cite[Corollary 5]{crapo1969mobius}. 

\begin{theorem} \label{th:gencrapo}
Let $A \subseteq B \subseteq \at(\mL)$, and let $S \subseteq \mL(B)$ be any subset. For all $x \in \mL(A)$ we have
$$\sum_{z \in \mL(A) \cap S} \mu_A(x,z)= \sum_{\substack{t \in \mL(B) \\ t^A=x}} \sum_{\substack{y \in S \\ y \ge t}}\mu_B(t,y).$$
\end{theorem}

\begin{proof}
For all $y \in S$ define 
$$\Sigma_y \defi \sum_{z \in \mL(A) \cap \{y\}} \ \sum_{\substack{t \in \mL(B) \\ t \ge z}} \mu_A(x,z) \; \mu_B(t,y).$$
By the properties of $\mu_B$ we have
\begin{equation} \label{ee.1}
\Sigma_y=\sum_{z \in \mL(A)} \mu_A(x,z) \sum_{\substack{t \in \mL(B) \\ z \le t \le y}}  \mu_B(t,y) = \sum_{z \in \mL(A) \cap \{y\}} \mu_A(x,z).
\end{equation}
Similarly, by Lemma \ref{lem:toA} and the properties of $\mu_A$ we have
\begin{equation}
\Sigma_y = \sum_{t \in \mL(B)} \mu_B(t,y) \sum_{\substack{z \in \mL(A) \\ z \le t}} \mu_A(x,z)
= \sum_{t \in \mL(B)} \mu_B(t,y) \sum_{\substack{z \in \mL(A) \\ z \le t^A}} \mu_A(x,z) 
= \sum_{\substack{t \in \mL(B) \\ t^A=x}} \mu_B(t,y). \label{ee.2}
\end{equation}
Combining (\ref{ee.1}) and (\ref{ee.2}) we see that, for all $y \in S$, 
$$\sum_{z \in \mL(A) \cap \{y\}} \mu_A(x,z) =  \sum_{\substack{t \in \mL(B) \\ t^A=x}} \mu_B(t,y).$$
Now sum over $y \in S$ and exchange the order of summation.
\end{proof}

The next result is a decomposition formula for the 
M\"{o}bius function of $\mL$ in terms of the M\"{o}bius functions of certain sublattices generated by atoms. Note that the result implies Whitney Theorem \cite[Proposition 3.11.3]{stanleyec} as a simple corollary.

\begin{theorem} \label{th:decomp}
Let $A_1,...,A_L \subseteq \at(\mL)$ be any subsets with $A_1 \cup \cdots \cup A_L = \at(\mL)$. For all $x \in \mL$ we have
\begin{equation} \label{eq:expr0}
\mu_\mL(x) = \sum_{ \substack{(x_1,...,x_L) \in \mL(A_1) \times \cdots \times \mL(A_L) \\ x_1 \vee \cdots \vee x_L=x}} \ \ \prod_{\ell=1}^L \mu_{A_\ell}(x_\ell).
\end{equation}
\end{theorem}

\begin{proof}
Consider the product lattice $\mL(A_1) \times \cdots \times \mL(A_L)$
endowed with the coordinatewise order. Define a map
$\varphi:\mL(A_1) \times \cdots \times \mL(A_L) \to \mL$ by
 $(x_1,...,x_L) \mapsto x_1 \vee \cdots \vee x_L$, and a map 
$\psi: \mL \to \mL(A_1) \times \cdots \times \mL(A_L)$ by 
$x \mapsto (x^{A_1},...,x^{A_L})$. Using Lemma \ref{lem:toA} one checks that $(\varphi,\psi)$ is a monotone Galois connection in the sense of Rota~\cite{rotaGC}. Applying \cite[Theorem 1]{rotaGC} in the form given by Greene 
\cite[p. 563]{greene82} we obtain that, for all $x \in \mL$,
\begin{equation} \label{iden}
\sum_{ \substack{(x_1,...,x_L) \in \mL(A_1) \times \cdots \times \mL(A_L) \\ x_1 \vee \cdots \vee x_L=x}} \ \ \prod_{\ell=1}^L \mu_{A_\ell}(x_\ell) \ = 
\sum_{\substack{t \in \mL \\ \psi(t)=(0,...,0)}} \mu_\mL(t,x).
\end{equation}
 Since $A_1 \cup \cdots \cup A_L=\at(\mL)$ and $\mL$ is atomistic, the only $t \in \mL$ with $\psi(t)=(0,...,0)$ is the zero element of $\mL$. Therefore the theorem follows from (\ref{iden}).
\end{proof}




We conclude this section with a slightly more involved decomposition of the M\"obius function of a lattice. 
 In passing, we will obtain a proof of Stanley's {Modular Factorization Theorem}~\cite[Theorem~2]{stanley1971modular}. 
The proof of the following result can be found in Appendix~\ref{app:A}.

\begin{theorem} \label{teo:av} 
Let $A, B \subseteq \at(\mL)$ be sets with $A \cup B= \at(\mL)$. For all $x \in \mL$ we have
$$\mu_\mL(x) = \sum_{\substack{x_A \in \mL(A), \: x_B \in \mL(B) \\ x_B^{A}=0 \\ x_A \vee x_B =x}} \mu_{A}(x_A) \: \mu_{B}(x_B).$$
\end{theorem}

 The previous theorem has the following consequence.

 \begin{corollary} \label{coro:three}
Suppose that $\mL$ is geometric and that $t \in \mL$ is modular. 
Let $[0,t]$ be the interval $\{t' \in \mL \mid 0 \le t' \le t\}$.
For all $B \subseteq \at(\mL)$ with $B \cup \{a \in \at(\mL) \mid a \le t\}=\at(\mL)$ we have
\begin{equation} \label{f2}
w_i(\mL) = \sum_{j=0}^i w_j([0,t]) \sum_{\substack{x \in \mL(B) \\ r_\mL(x)=i-j \\ x \wedge t =0}} \mu_B(x).
\end{equation}
\end{corollary}
 \begin{proof}
Define $A \defi \{a \in\at(\mL) \mid a \le t\}$.
 Since $\mL$ is atomistic, we have that $\mL(A)$ coincides with the interval $[0,t]$ in $\mL$. Moreover, for all $x \in \mL(B)$ we have $x^A=0$ if and only if $x \wedge t=0$.  Therefore to deduce (\ref{f2}) from Theorem \ref{teo:av} it suffices to show 
 that $\rk(y \vee x) = \rk(y) + \rk(x)$ for all $y \in [0,t]$ and all $x \in \mL$ with 
$x \wedge t=0$. This fact is known and not difficult to see. 
%
 \end{proof}

Note that Stanley's {Modular Factorization Theorem} can now be obtained from the definition of characteristic polynomial and 
 Corollary~\ref{coro:three}, taking $B=\at(\mL)$.

\begin{corollary}[Stanley] \label{coro:MFT}
Suppose that $\mL$ is geometric of rank $n$, and that $t \in \mL$ is modular. We have
\begin{equation*} 
\chi(\mL;\lambda) = \chi([0,t];\lambda) \left( \sum_{\substack{x \in \mL \\ x \wedge t=0}} \mu_\mL(x) \: \lambda^{n-\rk_\mL(t)-\rk_\mL(x)} \right).
\end{equation*}
\end{corollary}

As illustrated by Stanley, an important application of Corollary \ref{coro:MFT} is the computation of the characteristic polynomials of supersolvable geometric lattices \cite{stanley1971modular, stanley1972supersolvable}. 

\begin{remark} \label{rem:MFT}
Recall that a geometric lattice $\mL$ is \textbf{supersolvable} \label{supers} if it has a maximal chain of modular elements $0=t_0 \lessdot t_1 \lessdot \cdots \lessdot t_r=1$, where $r=\rk(\mL)$. Applying Corollary \ref{coro:MFT} to such a lattice with $t=t_{r-1}$ one obtains
$\chi(\mL;\lambda)=\chi([0,t_{r-1}])\cdot \left(\lambda - \#\{a \in \at(\mL) \mid a \not\le t_{r-1}\} \right)$.
This factorization method can be iterated, arriving at
$$\chi(\mL;\lambda) = \prod_{i=1}^r \left( \lambda - \#\{a \in \at(\mL) \mid a \le t_r, \; a \not\leq t_{r-1}\} \right).$$
The reader is referred to \cite[Section 3]{stanley1971modular} for more details.
Corollary \ref{coro:MFT} is a powerful tool for computing characteristic polynomials of lattices. It has been studied and extended by various 
authors; see e.g.~\cite{brylawski1975modular,brylawski1980several,greene82,hallam2015factoring,blass1997mobius}.
Unfortunately, most lattices studied in this paper (including higher-weight Dowling lattices) are not supersolvable in general. Moreover, their characteristic polynomials do not split into linear factors, not even over $\R$.
In particular, Stanely's Modular Factorization Theorem and its generalizations cannot be directly applied to compute them.
\end{remark}

\section{Whitney Numbers and Subspace Distributions}
\label{sec:wnsd}

In this section we establish some relations between Whitney numbers and subspace distributions. These will be applied in several instances to study geometric lattices that are not supersolvable, and whose characteristic polynomials do not split into linear factors.

In the sequel, we follow the notation of Section \ref{sec:prel} and work with a fixed space $X$ of dimension $n \ge 1$; see in particular Notation \ref{notaz X}.

In \cite[Theorem 2]{dowling1971codes}, Dowling expresses the number of $k$-dimensional spaces $V \le X$ that distinguish a certain set $A \subseteq X$ in terms of the evaluations of $\chi(\mL(A);\lambda)$ at prescribed integer values; see also \cite[Section~7.6]{zaslavsky1987mobius} and \cite[Section 4.4]{kung1996critical}. We start by refining this result and proving its converse. More precisely, we show that for all $A \subseteq X$ and for all $0 \le t \le n$, the following sets of numerical quantities determine each other:
\begin{itemize}[noitemsep]
\item the numbers $\{\alpha_k(X,A) \mid 0 \le k \le t \}$,
\item the Whitney numbers $\{w_i(\mL(A)) \mid 0 \le i \le t\}$.
\end{itemize}
We also provide explicit formul{\ae} that relate the two sets of numbers.
This shows that \textit{partial} information on the subspace distribution of $(X,A)$ yields \textit{partial} information on $\chi(\mL(A))$, and vice versa.
In later sections, we will use this observation to derive new properties of higher-weight Dowling lattices and compute their Whitney numbers in some cases; see Sections \ref{sec:HWDL3}, \ref{sec:HWDL4}, and \ref{sec:HWDL5}.

\begin{theorem} \label{prop2}
Let $A \subseteq X$ be a subset. For all $i \in \N$ we have
$$w_i(\mL(A)) = \sum_{k=0}^i \alpha_k(X,A) \qbin{n-k}{i-k}{q} (-1)^{i-k} q^{\binom{i-k}{2}}.$$ 
\end{theorem}

\begin{proof}
Fix $i \in \N$. We apply Theorem \ref{th:gencrapo} to $\mL(A)$ and $\mL(X)$, setting $x=0$ and taking as $S$ the set of subspaces of $X$ of dimension $i$. We obtain
\begin{equation} \label{intt}
\sum_{\substack{U \in \mL(A) \\ \rk(U)=i}} \mu_A(U) = \sum_{\substack{V \le X \\ V^A=\{0\}}} \ \sum_{\substack{Y \le X \\ \dim(Y)=i \\ Y \ge V}} \mu_X(V,Y).
\end{equation}
The left-hand side of (\ref{intt}) is $w_i(\mL(A))$, while the right-hand side can be re-written as
$$\sum_{k \in \N} \alpha_k(X,A) \qbin{n-k}{i-k}{q} (-1)^{i-k} q^{\binom{i-k}{2}}=\sum_{k=0}^i \alpha_k(X,A) \qbin{n-k}{i-k}{q} (-1)^{i-k} q^{\binom{i-k}{2}},$$
concluding the proof.
\end{proof}

The following result shows that $\alpha_k(X,A)$ only depends on the \textit{first} $k$ Whitney numbers of $\mL(A)$. It establishes the inverse of Theorem \ref{prop2}.

\begin{theorem} \label{prop1}
Let $A \subseteq X$ be a subset. For all $k \in \N$ we have
$$\alpha_k(X,A)= \sum_{i=0}^k w_i(\mL(A)) \qbin{n-i}{k-i}{q}.$$
\end{theorem}

\begin{proof} Fix $k \in \N$. Using Theorem \ref{prop2} we find
\begin{eqnarray*}
\sum_{i=0}^k w_i(\mL(A)) \qbin{n-i}{k-i}{q} &=& \sum_{j=0}^k \alpha_j(X,A) \sum_{i=j}^k \qbin{n-j}{i-j}{q} \qbin{n-i}{k-i}{q} (-1)^{i-j} q^{\binom{i-j}{2}}\\
&=&  \sum_{j=0}^k \alpha_j(X,A) \sum_{i=j}^k \qbin{n-j}{n-i}{q} \qbin{n-i}{n-k}{q} (-1)^{i-j} q^{\binom{i-j}{2}},
\end{eqnarray*}
where the latter equality follows from the first part of Lemma \ref{proprq} applied twice. Combining the the second part of the lemma with standard computations and the $q$-Binomial Theorem one obtains
\begin{equation*}
\sum_{i=0}^k w_i(\mL(A)) \qbin{n-i}{k-i}{q} = \alpha_k(X,A). \qedhere
\end{equation*}
\end{proof}

Using Theorem \ref{prop1} one can obtain a  formula for the number of $k$-dimensional spaces $V \le X$ that avoid a given linear space $A \le X$.
The result is very well-known but we include it for completeness.

\begin{corollary} \label{coro:av}
Let $A \le X$ be a linear subspace of dimension $a$. For all $k \in \N$, the number of $k$-dimensional subspaces of $X$ that distinguish $A$ is
$$\alpha_k(X,A)= \sum_{i=0}^k {(-1)}^i q^{\binom{i}{2}} \qbin{a}{i}{q}\qbin{n-i}{k-i}{q}.$$
\end{corollary}

As an application of Theorem \ref{prop2}, we now show that a lower bound on the critical exponent of $\mL(A)$ gives recursions for the Whitney numbers of the lattice $\mL(A)$. More precisely, the next result proves that the Whitney numbers of $\mL(A)$ are fully determined by the first $t$ of them, where $t=\rk(A)-\crit(\mL(A))$.

\begin{corollary}
Let $A \subseteq X$ be a subset. For all $i \in \N$ with $\rk(A)-\crit(\mL(A)) < i \le \rk(A)$ we have 
$$w_i(\mL(A)) = - \sum_{j=0}^{i-1} w_j(\mL(A)) \qbin{\rk(A)-j}{i-j}{q}.$$
In particular, the Whitney numbers of $\mL(A)$ are uniquely determined by
$\rk(A)$, $\crit(\mL(A))$, and the Whitney numbers $w_i(\mL(A))$ with
$0 \le i \le \rk(A)-\crit(\mL(A))$.
\end{corollary}
\begin{proof}
We shall assume $\rk(A)=n$ without loss of generality. By definition of critical exponent and Remark \ref{rem:imp} we have that 
$\alpha_i(X,A)=0$ for all $i>n-\crit(\mL(A))$. Therefore Theorem~\ref{prop1} gives
$$\sum_{j=0}^i w_j(\mL(A)) \qbin{n-j}{i-j}{q}=0$$
 for all $i$ with $n-\crit(\mL(A)) < i \le n$, from which the desired result follows.
\end{proof}

We can show a first concrete application of the results of this section. We will need the following definition from coding theory.

\begin{definition} \label{def:wH}
The \textbf{Hamming weight} of a vector $v \in \F_q^n$ is $\wH(v) \defi |\{1 \le i \le n \mid v_i \neq 0\}|$.
\end{definition}

Let $X=\F_2^n$ with $n \ge 1$, and denote by $\OO_n$ the set of vectors in $\F_2^n$ with odd Hamming weight.
The lattice $\mL(\OO_n)$ consists of all those subspaces of $\F_2^n$ (i.e., binary codes) that have a basis made of odd-weight vectors, ordered by inclusion. We call it the \textbf{odd-weight binary geometry} of order $n$. This lattice is not supersolvable in general, and its characteristic polynomial does not split into linear factors (see Example \ref{exnot}). We can nonetheless compute the Whitney numbers of this lattice (and therefore its characteristic polynomial) using Theorem~\ref{prop2}.

\begin{corollary}
For all $n \ge 1$, the Whitney numbers of the odd-weight binary geometry $\mL(\OO_n)$ are given by the formula
\begin{equation*}
w_i(\mL(\OO_n))= \sum_{k=0}^i \qbin{n-1}{k}{2} \qbin{n-k}{i-k}{2} (-1)^{i-k} \; 2^{\binom{i-k}{2}} \quad \mbox{for all } i \in \N.
\end{equation*}
\end{corollary}

\begin{proof}
Denote by $\EE_n$ the set of vectors in $\F_2^n$ with even weight. Since the sum of two vectors of even weight is again a vector of even weight,
the vectors $v \in \F_2^n$ of even weight form a linear subspace $Y \le \F_q^n$. Moreover, $\mL(\EE_n)$ coincides with $\mL(Y)$, the lattice of subspaces of $Y$.
It is easy to see that $\dim(Y)=n-1$. 
Observe moreover that the elements of $\mL(\EE_n)$ are precisely the even spaces, i.e., those subspaces of $\F_2^n$ that avoid $\OO_n$. As a consequence, 
\begin{equation} \label{pro3}
\alpha_k(\F_2^n,\OO_n) = 
\qbin{n-1}{k}{2} \quad \mbox{for all } k \in \N.
\end{equation}
Finally, combining Theorem~\ref{prop2} with~(\ref{pro3}) one computes 
the Whitney numbers of the odd-weight binary geometry, as desired.
%
\end{proof}


\begin{example} \label{exnot}
Let $n=3$. Then
$\chi(\mL(\OO_3);\lambda) = \lambda^3 - 4\lambda^2 + 6\lambda -3 = 
(\lambda-1)(\lambda^2-3\lambda+3)$. The polynomial $\lambda^2-3\lambda+3$ is irreducible over $\R$. In particular, $\mL(\OO_3)$ is not supersolvable.
\end{example}

The next result is very simple lower bound for the subspace distribution associated with $X$ and $A \subseteq X$. 
In Section \ref{sec:HWDL5} we will use this observation and Theorem~\ref{prop2} to obtain asymptotic estimates for the Whitney numbers of higher-weight Dowling lattices as the field size grows.

\begin{proposition} \label{prop:smalle}
Let $A \subseteq X$ be a set, and let $1 \le k \le n$ be an integer. 
We have 
$$\alpha_k(X,A) \ge \qbin{n}{k}{q} - |A| \cdot \qbin{n-1}{k-1}{q}.$$
\end{proposition}

\begin{proof} The $k$-dimensional subspaces $V \le X$ that intersect $A$ non-trivially are at most
\begin{equation*}
\sum_{v \in A \setminus \{0\}} |\{V \le X \mid \rk(V)=k, \, v \in V \}| = |A| \cdot \qbin{n-1}{k-1}{q}.
\qedhere
\end{equation*}
\end{proof}

The previous proposition gives the best bound when applied to a minimal sets of representatives, defined as follows.

\begin{definition} \label{def:msr}
Let $A \subseteq X$ be a subset. A \textbf{minimal set of representatives} for $A$ is a subset $B \subseteq A$ with the following properties:
\begin{itemize}[noitemsep]
\item $0 \notin B$,
\item for every $v \in A$ with $v \neq 0$ there exists $w \in B$ with $\langle v \rangle = \langle w \rangle$,
\item for every $v,w \in B$ with $v \neq w$, we have $\langle v \rangle \neq \langle w \rangle$.
\end{itemize} 
\end{definition}


\begin{remark} \label{rmk:msr}
Let $A \subseteq X$ be a subset, and let $B \subseteq A$ be a minimal set of representatives for~$A$. Then a subspace $V \le X$ distinguishes (resp., meets) $A$ if and only if it distinguishes (resp., meets)~$B$. In particular, we have $\alpha_k(X,A)=\alpha_k(X,B)$ for all $k \in \N$.
\end{remark}

We conclude this section with a characterization of the M\oo bius function of a lattice of the form~$\mL(A)$ in terms of subspace distributions. 
See Definition~\ref{def:alpha} for the notation.

\begin{theorem} \label{th:muexpl}
Let $A \subseteq X$ be a subset. For all $i \in \N$ and all $U \in \mL(A)$ with $\rk(U)=i$ we have
$$\mu_A(U) = \sum_{j=0}^{i} (-1)^{i-j} q^{\binom{i-j}{2}} \; \alpha(U,A,j).$$
\end{theorem}


\begin{proof}
The result is immediate if $A=\emptyset$ or $A=\{0\}$. We henceforth assume $A \neq \emptyset$ and $A \neq \{0\}$. Denote by $a_1,...,a_L$ the non-zero elements of $A$.
For all $1 \le \ell \le L$, let $A_\ell$ denote the 1-dimensional space generated by $a_\ell$. Moreover, for all $U \le X$ let
$$f(U) \defi \sum_{ \substack{(U_1,...,U_L) \in \mL(A_1) \times \cdots \times \mL(A_L) \\ U_1+ \cdots +U_L=U}} \mu_{A_1}(U_1) \cdots \mu_{A_L}(U_L),$$
where the sum over an empty index set is zero by definition.
Observe that for all $U \le X$ we have
$$g(U) \defi \sum_{U' \le U} f(U') = \left( \sum_{U_1 \le A_1 \cap U} \mu_{A_1}(U_1) \right) \cdots \left( \sum_{U_L \le A_L \cap U} \mu_{A_L}(U_L) \right).$$
The latter product is zero unless $A_\ell \cap U=\{0\}$ for all $1 \le \ell \le L$, in which case $g(U)=1$. Using M\oo bius inversion in $\mL(X)$ we then obtain
$$f(U) = \sum_{U' \le U} \mu_X(U',U) \; g(U') = \sum_{j=0}^{i} (-1)^{i-j} q^{\binom{i-j}{2}} \; \alpha(U,A,j) \quad \mbox{for all $U \le X$}.$$
Finally, observe that by Theorem \ref{th:decomp} we have
$\mu_A(U)=f(U)$ for all $U \in \mL(A)$. 
\end{proof}

\section{Distinguishing Multiple Spaces and Orthogonality} \label{sec:dms}

Corollary \ref{coro:av} provides a formula for the number of $k$-dimensional subspaces $V \le X$ that distinguish a given subspace $A \le X$. In this section consider the number of subspaces of dimension $k$ that distinguish 
an $L$-tuple $A_1,...,A_L$ of subspaces of $X$, i.e., that distinguishes $A_1 \cup \cdots \cup A_L$.


\begin{remark}
It follows from the definitions and Remark \ref{rem:imp} that
$\crit(\mL(A_1 \cup \cdots \cup A_L)) \ge \max_\ell \dim(A_\ell)$.
Equivalently, the integers of the form $q^i$, for $0 \le i < \max_\ell \dim(A_\ell)$,
are roots of the characteristic polynomial $\chi(\mL(A_1 \cup \cdots \cup A_L);\lambda)$.
\end{remark}

When $L=2$, one can obtain detailed information about the corresponding lattice $\mL(A_1 \cup A_2)$. For lattices of this form, the following theorem explicitly computes their Whitney numbers. The proof uses Theorem \ref{teo:av} and can be found in Appendix \ref{app:A}.

\begin{theorem} \label{twospacesS}
Let $A_1, A_2 \le X$ be subspaces. 
For all $i \in \N$ we have
\begin{equation*} 
w_i(\mL(A_1 \cup A_2)) = \sum_{j=0}^{i} \ \sum_{h=0}^{i-j} (-1)^{i+h} \; q^{\binom{j}{2} + \binom{i-j}{2} + \binom{h}{2}} 
  \qbin{\rk(A_1)}{j}{q} \qbin{\rk(A_1 \cap A_2)}{h}{q} \qbin{\rk(A_2)-h}{i-j-h}{q}.
\end{equation*}
\end{theorem}

A particularly interesting case is when all the spaces $A_1,...,A_L \le X$ have the same dimension, say~$a$. In this situation we have $\crit(\mL(A_1 \cup \cdots \cup A_L)) =a$ if and only if the spaces $A_1,...,A_L$ share a common complement in the lattice $\mL(X)$ of subspaces of $X$.
In the reminder of the section we concentrate on the equidimensional case, and prove a partial duality theorem between the Whitney numbers of $\mL(A_1 \cup \cdots \cup A_L)$ and $\mL(A_1^\perp \cup \cdots \cup A_L^\perp)$, where
$A_\ell^\perp$ denotes the orthogonal of $A_\ell$ (defined below). This result will be applied later in Sections \ref{sec:HWDL2} and \ref{sec:HWDL5} to derive some properties of the Whitney numbers of higher-weight Dowling lattices.

\begin{notation}
In the sequel, we fix a symmetric, non-degenerate bilinear form $b : X \times X \to \F_q$. For $V \le X$ we let $V^\perp \defi \{x \in X \mid b(x,y)=0 \mbox{ for all } y \in V\} \le X$ denote the \textbf{orthogonal} of 
$V$ with respect to $b$.
None of the results in this section depends on the specific choice of~$b$.
\end{notation}

The next lemma summarizes some well-known properties of
orthogonal spaces.  

\begin{lemma} \label{dualiz} \label{relation}
Let $U,V \le X$ be subspaces. The following hold.
\begin{enumerate}[label=(\arabic*), noitemsep]
\item $\rk(V^\perp)=n-\rk(V)$.
\item $V^{\perp \perp}=V$.
\item $(U+V)^\perp=U^\perp \cap V^\perp$ and 
$(U\cap V)^\perp=U^\perp + V^\perp$.
\item $\rk(U \cap V) = \rk(U)-n+\rk(V)+\rk(U^\perp \cap V^\perp)$.
\end{enumerate}
\end{lemma}

The first three properties above are straightforward. The fourth can be deduced from the identity
$(U\cap V)^\perp =U^\perp +V^\perp$ taking dimensions.

It is natural to ask how the Whitney numbers of $\mL(A_1 \cup \cdots \cup A_L)$ and $\mL(A_1^\perp \cup \cdots \cup A_L^\perp)$ relate to each other. A partial answer to this question is given by the following result. We will apply it later to higher-weight Dowling lattices.

\begin{theorem} \label{thm:dua}
Let $A_1,...,A_L \le X$ be subspaces, all of which have the same dimension $a$. Then
$$\sum_{i=0}^{n-a} w_i(\mL(A_1 \cup \cdots \cup A_L)) \qbin{n-i}{a}{q} = \ 
\sum_{i=0}^{a} w_i(\mL(A_1^\perp \cup \cdots \cup A_L^\perp)) \qbin{n-i}{a-i}{q}.$$
\end{theorem}

\begin{proof}
Suppose that $V \le X$ is a space of dimension $n-a$ with $\rk(V \cap A_\ell)=0$ for all $1 \le \ell \le L$. Then by Lemma \ref{relation} we have that $V^\perp \le X$ is a space of dimension $a$ such that $\rk(V^\perp \cap A_\ell^\perp)=0$ for all $1 \le \ell \le L$. Similarly, if $V \le X$ is a space of dimension $a$ with $\rk(V \cap A^\perp_\ell)=0$ for all $1 \le \ell \le L$, then 
$V^\perp \le X$ is a space of dimension $n-a$ such that $\rk(V^\perp \cap A_\ell)=0$ for all $1 \le \ell \le L$. All of this shows that 
 the map $V \mapsto V^\perp$ is a bijection 
\begin{multline*}
\{V \le X \mid \rk(V)=n-a, \ \rk(V \cap A_\ell)=0 \mbox{ for all } 1 \le \ell \le L\} \to \\ \to
\{V \le X \mid \rk(V)=a, \ \rk(V \cap A^\perp_\ell)=0 \mbox{ for all } 1 \le \ell \le L\}.
\end{multline*}
In particular, $\alpha_{n-a}(X,A_1 \cup \cdots \cup A_L) = \alpha_a(X,A_1^\perp \cup \cdots \cup A_L^\perp)$.
Applying Theorem \ref{prop1} to $A_1 \cup \cdots \cup A_L$ we obtain
$$
\alpha_{n-a}(X,A_1 \cup \cdots \cup A_L) = \sum_{i=0}^{n-a} w_i(\mL(A_1 \cup \cdots \cup A_L)) \qbin{n-i}{n-a-i}{q}.
$$
Applying the same Theorem to $A_1^\perp \cup \cdots \cup A_L^\perp$ we get
$$
\alpha_a(X,A_1^\perp \cup \cdots \cup A_L^\perp) = \sum_{i=0}^a w_i(\mL(A_1^\perp \cup \cdots \cup A_L^\perp)) \qbin{n-i}{a-i}{q}.$$
The result follows.
\end{proof}

In the remainder of the paper we concentrate on a class of lattices introduced by Dowling in 1971, and apply the results of previous sections to study various aspects of their general theory.

\section{HWDL -- Introduction, Definitions, and First Properties}
\label{sec:HWDL1}

%
%
%
%
%
%
%
%
%
%

The next seven sections are devoted to a special class of combinatorial geometries, known as \textit{higher-weight Dowling lattices} (HWDL in short). These were introduced in 1971 by Dowling \cite{dowling1971codes}, in connection to fundamental problems in coding theory. They were further studied, among others, by Zaslavsky \cite{zaslavsky1987mobius}, Bonin \cite{bonin1993modular,bonin1993automorphism}, Kung \cite{kung1996critical}, 
Brini \cite{brini1982some}, and Games \cite{games1983packing}. 

To date, still very little is known about higher-weight Dowling lattices, and the techniques for studying them have not been discovered yet  \cite[Section 7.6]{zaslavsky1987mobius}. 

In the remainder of the paper we apply the results of the previous sections to bring forward the theory of higher-weight Dowiling lattices. We study their general properties, give explicit formul{\ae} for some of their Whitney numbers, obtain bounds on the asymptotic growth of these, and discuss their polynomiality in $q$. 
In passing, we also obtain results on the enumerative combinatorics of error-correcting codes endowed with the Hamming metric.

In this first section we recall the basic definitions, describe the connection between higher-weight Dowling lattices and error-correcting codes, and survey the main known results on the subject obtained by Dowling, Bonin, and Kung.

\begin{definition} \label{def:HWDL}
Given a prime power $q$ and integers $n,d \ge 1$, 
we denote by $\HH(q,n,d)$ the set of vectors $v \in \F_q^n$ with $1 \le \wH(v) \le d$, where $\wH$ is the Hamming weight; see Definition \ref{def:wH}. 
The \textbf{higher-weight Dowling lattice} associated to $(q,n,d)$ is
$$\mH(q,n,d) \defi \mL(\HH(q,n,d)).$$
In other words, $\mH(q,n,d)$ is the geometric sublattice of $\mL(\F_q^n)$ whose atoms are the 1-dimensional subspaces of $\F_q^n$ generated by a vector $v$ of Hamming weight $1 \le \wH(v) \le d$.
\end{definition}

\begin{notation}
For simplicity, the M\oo bius function of $\mH(q,n,d)$ is denoted by $\mu_{q,n,d}$, and the $i$-th Whitney number (of the first kind) of $\mH(q,n,d)$ is denoted by $w_i(q,n,d)$. Finally, for all integers $k \ge 0$ we let
$$\alpha_k(q,n,d) \defi \alpha_k(\F_q^n,\HH(q,n,d)).$$
In the sequel, $q$ always denote a prime power.
\end{notation}

We also briefly recall the needed coding theory terminology.

\begin{definition}
 A \textbf{code} is an $\F_q$-linear subspace $C \le \F_q^n$. The \textbf{minimum Hamming distance} of a non-zero code $C \le \F_q^n$ is the integer
$\dH(C) \defi \min\{\wH(v) \mid v \in C, \, v \neq 0\}$. We let $+\infty$ be the minimum Hamming distance of the zero code $\{0\} \le \F_q^n$.
\end{definition}

The following facts follow easily from the definitions.

\begin{remark} \label{frems}
Let $n,d \ge 1$ be integers.
\begin{enumerate}[label=(\arabic*), noitemsep]
\item The elements of $\mH(q,n,d)$ are those subspaces of $\F_q^n$ that have a basis made of vectors of Hamming weight $\le d$. 
\item For all integers $k \ge 0$, the number of $k$-dimensional codes $C \le \F_q^n$ with minimum Hamming distance $d_\HH(C)>d$ is 
$\alpha_k(q,n,d)$.
\item $\mH(q,n,1)$ is isomorphic to the Boolean algebra over the set $\{1,...,n\}$, for any $q$.
\item If $d \ge n$, then $\mH(q,n,d)$ coincides with the lattice of subspaces of $\F_q^n$.
\end{enumerate}
\end{remark}

The next result explains the terminology ``higher-weight Dowling lattice''. See for example~\cite[page~134]{zaslavsky1987mobius}, \cite{dowling1973q}, or \cite{dowling1973class}.

\begin{theorem}[Dowling] \label{th:down2}
The lattice $\mH(n,2)$ is isomorphic to the Dowling lattice 
$Q_n(\F_q^*)$, where~$\F_q^*$ denotes the multiplicative group of $\F_q$.
\end{theorem}

The connection between  coding theory and geometric lattices generated by small-weight vectors was first pointed out by Dowling. In \cite[Section 5]{dowling1971codes}, he proved that the number of codes $C \le \F_q^n$ of dimension $k$ and minimum Hamming distance $\dH(C) >d$ can be expressed in terms of the evaluation of the polynomial $\chi(\mH(q,n,d);\lambda)$ at prescribed integer values. See also \cite[Proposition~7.6.6]{zaslavsky1987mobius} and \cite[Proposition 4.13]{kung1996critical}. We can make this connection more precise using Theorems~\ref{prop2} and~\ref{prop1}, showing that the problem of computing the number of such codes is in fact \textit{equivalent} to that of computing the \textit{first}~$k$ Whitney numbers of the lattice~$\mH(q,n,d)$.

\begin{corollary} \label{coro:relations}
Let $n \ge d \ge 1$ be integers. The following hold.
\begin{enumerate}[label=(\arabic*), noitemsep]
\item There is a code $C \le \F_q^n$ of dimension $k$ with $d_\HH(C) >d$ if and only if $\chi(\mH(q,n,d);q^{n-k}) \neq 0$, i.e., if and only if $\crit(\mH(q,n,d)) \le n-k$.
\item For all $k \ge 0$, the number of codes $C \le \F_q^n$ of dimension $k$  with $d_\HH(C) >d$ is
$$\alpha_k(q,n,d) = \sum_{i=0}^k w_i(q,n,d) \qbin{n-i}{k-i}{q}.$$
\item For all $i \ge 0$, the $i$-th Whitney number of $\mH(q,n,d)$ is given by
\begin{equation*}
w_i(q,n,d)=\sum_{k=0}^i \alpha_k(q,n,d) \qbin{n-k}{i-k}{q} (-1)^{i-k} q^{\binom{i-k}{2}}.
\end{equation*}
\end{enumerate}
\end{corollary}
\begin{proof}
The three statements follow, respectively, from Remark \ref{rem:imp}, Theorems~\ref{prop1} and \ref{prop2}.
\end{proof}

Recall from the Singleton Bound \cite{singleton1964maximum} that any code $C \le \F_q^n$ of dimension $k \ge 1$ satisfies
$d_\HH(C) \le n-k+1$. By Remark \ref{rem:imp}, we can re-state this result as follows.
\begin{theorem}[Singleton] \label{th:sing}
Let $n \ge d \ge 1$ be integers. We have $\crit(\mH(q,n,d)) \ge d$.
\end{theorem}

 Codes meeting the Singleton Bound are called \textbf{MDS} ({Maximum Distance Separable}). 
Determining the tuples $(n,k,q)$ for which there exists a $k$-dimensional MDS code $C \le \F_q^n$ is an open problem since the 50's. In particular, the following conjecture by Segre is open since 1955; see \cite{segre1955curve}.

\begin{conjecture}[Segre] \label{conj:segre}
Let $n \ge d \ge 1$ be integers with $2 \le k \le q$ and $d=n-k+1$. Suppose that $\alpha_k(q,n,d) \neq 0$. The following hold.
\begin{enumerate}[noitemsep,label=(\arabic*)]
\item If $q \equiv 0 \mod 2$ and $k \in \{3,q-1\}$, then $n \le q+2$.
\item In all other cases, $n \le q+1$.
\end{enumerate}
\end{conjecture}
Conjecture \ref{conj:segre} is known as the ``MDS Conjecture''. It is often referred to as the main problem in classical coding theory.
By Corollary \ref{rem:imp}, Segre's Conjecture is an instance of the Critical Problem for higher-weight Dowling lattices: Computing $\crit (\mH(q,n,d))$ for all values of $q,n,d$ would lead to solving this fundamental problem.

Unfortunately, standard methods for studying characteristic polynomials of lattices (such as Stanley's {Modular Factorization Theorem} and its generalizations) do not apply to higher-weight Dowling geometries.
In fact, Bonin showed that higher-weight Dowling lattices are not supersolvable, with only very few exceptions; see 
\cite[page 8, right after Lemma 3.2]{bonin1993modular}. 

\begin{theorem}[Bonin] \label{bonin1}
The following are the only supersolvable higher-weight Dowling lattices.
\begin{itemize}[noitemsep]
\item $\mH(q,n,1)$ for all $q$ and $n \ge 1$,
\item $\mH(q,n,2)$ for all $q$ and $n \ge 1$,
\item $\mH(n,n-1)$ for all $q$ and $\ge 2$,
\item $\mH(q,n,d)$ for all $q$ and $d \ge n$.
\end{itemize}
\end{theorem} 

Theorem \ref{bonin1} shows that most lattices of the form $\mH(q,n,d)$ are not supersolvable. In the following Example \ref{exroots}, we show that in general their characteristic polynomials do not split into linear factors, not even over $\R$.

\begin{example} \label{exroots}
Let $(q,n,d)=(3,6,3)$. Then the characteristic polynomail of the lattice $\mH(q,n,d)$ is $(\lambda-1)(\lambda-3)(\lambda-9)(\lambda-27)(\lambda^2-76\lambda+1515) \in \Z[\lambda]$. As $76^2-4\cdot 1515=-284$, the polynomial 
$\lambda^2-76\lambda+1515$ is irreducible over $\R$.
\end{example}

Another interesting fact is that the property of splitting into linear factors heavily depends on the value of $q$. More precisely, one can find pairs $(q,n,d)$ and $(q',n,d)$ for which $\chi(\mH(q,n,d);\lambda)$ splits into linear factors over $\Z$, while 
$\chi(\mH(q',n,d);\lambda)$ does not (not even over $\R$). We illustrate this in the next example.

\begin{example}
Table \ref{tabb} contains the factorization over $\R$ of the characteristic polynomials associated to the parameters
$(2,5,3)$, $(3,5,3)$, and $(4,5,3)$.
\begin{table}[H]
\centering
\begin{tabular}{|c|c|}
\hline
Parameters $(q,n,d)$ & Factorization of $\chi(\mH(q,n,d);\lambda) \in \Z[\lambda]$ over $\R$ \\
\hline 
\hline 
$(2,5,3)$ &  $(\lambda-1)(\lambda-2)(\lambda-4)(\lambda-8)(\lambda-10)$ \\
\hline 
$(3,5,3)$ &   $(\lambda-1)(\lambda-3)(\lambda-9)(\lambda-25)(\lambda-27)$  \\
\hline
$(4,5,3)$ & $(\lambda-1)(\lambda-4)(\lambda-16)(\lambda^2-104\lambda+2722)$ \\
\hline
\end{tabular}
\caption{The characteristic polynomials of some higher-weight Dowling lattices.}
\label{tabb}
\end{table}
\end{example}

According to Theorem \ref{bonin1}, the lattice $\mH(q,n,n-2)$ is not supersolvable for $n \ge 5$. However, its characteristic polynomial is known. It has been computed by Bonin \cite[Theorem 3.5]{bonin1993modular}.

\begin{theorem}[Bonin] \label{th:bonin2}
Suppose $n \ge 3$. The characteristic polynomial of $\mH(q,n,n-2)$ is given by the formula
\begin{multline*}\chi(\mH(q,n,n-2);\lambda)= \biggl( (\lambda-q^{n-2})(\lambda-q^{n-1}) + (\lambda-q^{n-2}) \bigl( (q-1)^{n-1}+n(q-1)^{n-2} \bigr) \\ + (q-1)^n (q-2) \cdots (q-n+2)  \biggr) \prod_{i=0}^{n-3} (\lambda-q^i).
\end{multline*}
\end{theorem}

\begin{remark}
The approach of \ref{bonin1} has been extended by Kung to study higher-weight Dowling lattices $\mH(q,n,d)$ whose critical exponent is $n-2$; see 
\cite[Corollary 6.11]{kung1996critical} for the precise statement.
\end{remark}

\section{HWDL -- New General Properties}
\label{sec:HWDL2}

In this section we establish some new general properties of the Whitney numbers of higher-weight Dowling lattices. These will be applied in later sections in more concrete contexts.

Our first result shows that, if $n \ge id$, then the value of $w_i(q,n,d)$ only depends on the Whitney numbers $w_i(q,t,d)$ with $1 \le t \le id-1$. More precisely, the following hold.

\begin{theorem} \label{th:compl}
Let $n,i,d \ge 1$ with $i \le n$. Suppose $n \ge id$. Then 
\begin{multline*}
w_i(q,n,d) \ = \ (-1)^i  (q-1)^{i(d-1)} \sum_{1 \le \ell_1< \ell_2 < \cdots < \ell_i \le n-d+1} \quad \left(  \prod_{j=1}^i  \binom{n-\ell_j-d(i-j)}{d-1} \right)  
\\ + \sum_{t=i}^{id-1} \binom{n}{t} w_i(q,t,d)\sum_{s=t}^{id-1} \binom{n-t}{n-s} (-1)^{s-t}.
\end{multline*}
\end{theorem}

Theorem \ref{th:compl} will find various applications in the remainder of the paper: It will allow us to explicitly compute $w_2(q,n,3)$ for all $q$ and all $n \ge 6$, and to establish an upper bound on the growth rate of $|w_i(q,n,d)|$ as $q \to +\infty$.  See Theorems \ref{th:comput3} and \ref{th:asnlarge} for details.
The proof of Theorem \ref{th:compl} requires some preliminary definitions and results.

\begin{definition}
The \textbf{Hamming support} of a vector $v \in \F_q^n$ is 
$\sH(v) \defi \{1 \le i \le n \mid v_i \neq 0\}$. The \textbf{Hamming support}
of a subspace $U \le \F_q^n$ is the set 
$$\sH(U) \defi \bigcup_{u \in U} \sH(u).$$
\end{definition}

\begin{notation}
For a non-empty set $S \subseteq \{1,...,n\}$ of cardinality $s$, we denote by $\pi_S:\F_q^n \to \F_q^s$ the projection on the coordinates indexed by $S$, i.e., 
$\pi_S: v \mapsto (v_i \mid i \in S) \in \F_q^s$ for all $v \in \F_q^n$.
\end{notation}

We start with the following simple preliminary result.

\begin{lemma} \label{lem:pi}
Let $n \ge d \ge 1$ be integers, and let $S \subseteq \{1,...,n\}$ be a set of cardinality $s \ge 1$. For all $U \in \mH(q,n,d)$ with $\sH(U) \subseteq S$ we have
$$\mu_{q,n,d}(U) = \mu_{q,s,d}(\pi_S(U)).$$
\end{lemma}

\begin{proof}
Since $\sH(U) \subseteq S$, the projection $\pi_S$ is a lattice isomorphism between the interval $[0,U]$ in $\mH(q,n,d)$ and the interval $[0,\pi_S(U)]$ in $\mH(q,s,d)$. In particular, the two intervals (as lattices) must have the same Euler characteristic.
\end{proof}

The second tool that we need is an explicit formula for the number of spaces $U \in \mH(q,n,d)$ that satisfy some special properties. 

\begin{proposition} \label{tecnlm}
Let $i,d \ge 1$ with $1 \le i \le n$, and suppose $n \ge id$. The number of spaces  $U \in \mH(q,n,d)$ with $\rk(U)=i$ and $|\sH(U)|=id$ is
$$(q-1)^{i(d-1)}\sum_{1 \le \ell_1< \ell_2 < \cdots < \ell_i \le n-d+1} \quad \prod_{j=1}^i \binom{n-\ell_j-d(i-j)}{d-1}.$$
Moreover, for any such $U$ we have 
$$\mu_{q,n,d}(U)=(-1)^i.$$
\end{proposition}

\begin{notation}
In the remainder of the section, we denote by $\lead(u)$ the \textbf{leading position} of a vector $u \in \F_q^n \setminus\{0\}$, and by $\inn(u)$ its \textbf{initial entry}. More precisely, we set $\lead(u)=\min\{i \mid u_i \neq 0\}$ and $\inn(u)=u_i \in \F_q$, where $i=\lead(u)$. Moreover, we simply write $[n]$ for $\{1,...,n\}$.
\end{notation}

\begin{proof}[Proof of Proposition \ref{tecnlm}]
Let $U \in \mH(q,n,d)$ be a space with $\rk(U)=i$ and $|\sH(U)|=id$. Let $\{u_1,...,u_i\}$ be a basis of $U$ made of vectors of weight at most $d$. Since $U$ has support size $id$, the $u_j$'s are all of weight exactly $d$ and must have pairwise disjoint supports. By permuting these vectors and scaling them we find a basis $\{u_1,...,u_i\}$ of $U$ with the following properties:
\begin{enumerate}[label=(\arabic*), noitemsep]
\item $\lead(u_1) < \lead(u_2) < \cdots < \lead(u_i)$, \label{cond1}
\item $\inn(u_j)=1$ for all $j$. \label{cond2}
\item $\sH(u_j) \cap \sH(u_k) = \emptyset$ if $j \neq k$, \label{cond3}
\item $\wH(u_j)=d$ for all $j$. \label{cond4}
\end{enumerate}
Note that the matrix whose rows are the $u_j$'s is in
reduced row-echelon form. Moreover, any matrix 
$$G=\begin{pmatrix} u_1 \\ \vdots \\ u_i\end{pmatrix} \in \F_q^{i \times n}.$$
whose rows satisfy \ref{cond1}, \ref{cond2}, \ref{cond3} and \ref{cond4}
is in reduced row-echelon form and generates a space $U \in \mH(q,n,d)$ of rank $i$ and support size $id$.
Therefore it suffices to count these matrices. 

Fix a choice $1 \le \ell_1< \ell_2 < \cdots < \ell_i \le n-d+1$ for the leading entries of the rows of $G$. Then we have 
$$\binom{n-\ell_i}{d-1}(q-1)^{d-1}$$ choices for the $i$-th row. For each of these choices, we have
$$\binom{n-\ell_{i-1}-d}{d-1}(q-1)^{d-1}$$ choices for the $(i-1)$-th row. Once the $(i-1)$-th row is picked, there are
$$\binom{n-\ell_{i-2}-2d}{d-1}(q-1)^{d-1}$$
choices for the $(i-2)$-th row. Continuing in this way we obtain that the number of $i \times n$ matrices $G$ whose rows satisfy
\ref{cond1}, \ref{cond2}, \ref{cond3} and \ref{cond4} and have leading entries $1 \le \ell_1< \ell_2 < \cdots < \ell_j \le n-d+1$ are
$$(q-1)^{i(d-1)} \; \prod_{j=1}^i \binom{n-\ell_j-d(i-j)}{d-1}.$$
Therefore the number of spaces  $U \in \mH(q,n,d)$ with $\rk(U)=i$ and $|\sH(U)|=id$ is
$$\sum_{1 \le \ell_1< \ell_2 < \cdots < \ell_i \le n-d+1} \quad (q-1)^{i(d-1)} \; \prod_{j=1}^i \binom{n-\ell_j-d(i-j)}{d-1},$$
as claimed.

It remains to show that $\mu_{q,n,d}(U)=(-1)^i$ for all spaces $U \in \mH(q,n,d)$ with $\rk(U)=i$ and $|\sH(U)|=id$. This can be seen in various ways. One is the following. Fix $U$ with the desired properties, and let
$\{u_1,...,u_i\}$ be a basis of $U$ that satisfies \ref{cond1}, \ref{cond2}, \ref{cond3} and \ref{cond4}. The only vectors $u \in U$ with $1 \le \wH(u) \le d$ are the $u_j$'s and their multiples. Therefore the interval $[0,U]$ in $\mH(q,n,d)$ is the lattice $\mL(\{u_1,...,u_i\})$. Moreover, it is not difficult to see that $\mL(\{u_1,...,u_i\})$ is isomorphic to the Boolean algebra $\mB_i$ over the set $\{1,...,i\}$. Thus
$\mu_{q,n,d}(U)$ coincides with the Euler characteristic of $\mB_i$, which is $(-1)^i$.
\end{proof}

We are now ready to prove the first main result of this section.

\begin{proof}[Proof of Theorem \ref{th:compl}] \label{pageth:compl}
Fix any $1 \le i \le n$. Observe that every space $U \in \mH(q,n,d)$ of dimension $i$ has $|\sH(U)| \le id$.
Therefore
\begin{equation} \label{sssum}
w_i(q,n,d) = \sum_{\substack{U \in \mH(q,n,d) \\ \rk(U)=i}} \mu_{q,n,d}(U)  
= \sum_{s=0}^{id} \ \ \sum_{\substack{S \subseteq [n] \\ |S|=s}} \ \sum_{\substack{U \in \mH(q,n,d) \\ \rk(U)=i \\ \sH(U)=S}} \mu_{q,n,d}(U)
=\sum_{s=0}^{id} \ \ \sum_{\substack{S \subseteq [n] \\ |S|=s}} f(S),
\end{equation}
where 
$$f(S) \defi  \sum_{\substack{U \in \mH(q,n,d) \\ \rk(U)=i \\ \sH(U)=S}} \mu_{q,n,d}(U) \quad \mbox{for all $S \subseteq [n]$.}$$ 
We will evaluate $f(S)$ in two different ways depending on the cardinality of $S$. Observe first that by Proposition \ref{tecnlm} we have
\begin{equation}\label{ssse1}
\sum_{\substack{S \subseteq [n] \\ |S|=id}} f(S) =  (-1)^i(q-1)^{i(d-1)} \sum_{1 \le \ell_1< \ell_2 < \cdots < \ell_i \le n-d+1} \quad \prod_{j=1}^i  \binom{n-\ell_j-d(i-j)}{d-1}.
\end{equation}
Now fix a subset $S \subseteq [n]$ with $|S| \le id-1$.
For all $T \subseteq S$ define
$$g(T) \defi \sum_{T' \subseteq T} f(T').$$ 
Since $i \ge 1$, we have $g(\emptyset)=0$. Moreover, by Lemma \ref{lem:pi}, for all $T \subseteq S$ with $|T|=t \ge 1$ we have
$$g(T)=\sum_{\substack{U \in \mH(q,n,d) \\ \rk(U)=i \\ \sH(U) \subseteq T}} \mu_{q,n,d}(U) = w_i(q,t,d).$$
Using M\oo bius inversion we conclude that
\begin{equation} \label{ssse2}
f(S)=\sum_{T \subseteq S} g(T) (-1)^{s-|T|} = \sum_{t=1}^s w_i(q,t,d) \binom{s}{t} (-1)^{s-t}, \quad \mbox{where } s=|S|.
\end{equation}
Combining (\ref{sssum}) with (\ref{ssse2}) one obtains
\begin{align*}
w_i(q,n,d) -  \sum_{\substack{S \subseteq [n] \\ |S|=id}} f(S)= \sum_{s=0}^{id-1}  \ \ \sum_{\substack{S \subseteq [n] \\ |S|=s}} f(S) 
 &=\sum_{s=0}^{id-1} \ \ \sum_{\substack{S \subseteq [n] \\ |S|=s}} \ \sum_{t=1}^s w_i(q,t,d) \binom{s}{t} (-1)^{s-t} 
\\ &= \sum_{s=1}^{id-1} \sum_{t=1}^s \binom{n}{s} \binom{s}{t}(-1)^{s-t}w_i(q,t,d)  \\
&= \sum_{t=1}^{id-1} \binom{n}{t} \sum_{s=t}^{id-1} \binom{n-t}{n-s} (-1)^{s-t} w_i(q,t,d).
\end{align*}
Note moreover that $w_i(q,t,d)=0$ for $t \le i-1$, as the lattice $\mH(q,t,d)$ has rank $t$. We then conclude using Eq.~(\ref{ssse1}).
\end{proof}

%

The second main result of this section establishes a (partial) duality between the Whitney numbers of $\mH(q,n,d)$ and $\mH(q,n,n-d)$. 

\begin{theorem} \label{thm:dual_appl}
 For all $n \ge d \ge 1$, the Whitney numbers of $\mH(q,n,d)$ and $\mH(q,n,n-d)$ satisfy 
$$\sum_{i=0}^{n-d} \qbin{n-i}{d}{q} w_i(q,n,d)  \ = \ 
\sum_{i=0}^{d} \qbin{n-i}{d-i}{q} w_i(q,n,n-d).$$
\end{theorem}
\begin{proof}
Let $S_1,...,S_L$ be the $d$-subsets of $\{1,...,n\}$, where $L=\binom{n}{d}$. For $1 \le \ell \le L$
define the space
$$A_\ell =\F_q^n(S_\ell)=\{v \in \F_q^n \mid v_i=0 \mbox{ for all $i \notin S_\ell$}\} \le \F_q^n.$$
Observe that $\mH(q,n,d)=\mL(A_1 \cup \cdots \cup A_L)$.
Denote by $A_\ell^\perp$ the orthogonal of $A_\ell$
with respect to the standard inner product of $\F_q^n$. It is easy to see that
$A_\ell^\perp=\F_q^n(S_\ell')$, where $S_\ell'$ is the complement of $S_\ell$ in $\{1,...,n\}$. As $S_\ell$ ranges over the $d$-subsets of $\{1,...,n\}$, $S_\ell'$
ranges over the $(n-d)$-subsets of $\{1,...,n\}$. Thus
$\mL(A_1^\perp \cup \cdots \cup A_L^\perp)=\mH(q,n,n-d)$, and we conclude by Theorem~\ref{thm:dua}.
\end{proof}

\section{HWDL -- Computational Results}
\label{sec:HWDL3}

In this section, we apply the results of this paper to obtain closed formul{\ae} for the Whitney numbers of $\mH(q,n,d)$ for certain values of the parameters $n,d$ and for any $q$.

\begin{theorem} \label{th:n-1}
Let $n, i \ge 2$ be integers. The $i$-th Whitney number of $\mH(q,n,n-1)$ is
\begin{equation*}
w_i(q,n,n-1)=\qbin{n-1}{i}{q}(-1)^iq^{\binom{i}{2}}  -\qbin{n-1}{i-1}{q}(-1)^{i-1}q^{\binom{i-1}{2}} \;  \sum_{j=1}^{n-1} \binom{n-1}{j-1} (q-1)^{j-1}.
\end{equation*}
\end{theorem}

\begin{proof}
We use Corollary \ref{coro:three}. Let
$T \le \F_q^n$ be the space generated by the basis elements $e_1,...,e_{n-1}$. It is not difficult to see that every subspace of $T$ belongs to $\mH(q,n,n-1)$. In particular, $T$ is modular in $\mH(q,n,n-1)$, and $U \meet T=U \cap T$ for all $U \in \mH(q,n,d)$. Moreover, an element $U \in \mH(q,n,n-1)$ with $U \cap T=\{0\}$ has either $\rk(U)=0$, or $\rk(U)=1$. Therefore by Corollary~\ref{coro:three} we conclude
\begin{equation} \label{iii} w_i(q,n,n-1)=w_i([0,T])-w_{i-1}([0,T]) \cdot \left| \{U \in \mH(q,n,n-1) \mid \rk(U)=1, \ U \not\le T\}\right|,
\end{equation}
where the square brackets denote an interval of $\mH(q,n,n-1)$. Since $[0,T]$ is isomorphic to the lattice of subspaces of~$T$, we have
\begin{equation} \label{iiii}
w_i([0,T])=\qbin{n-1}{i}{q}(-1)^iq^{\binom{i}{2}}, \qquad 
w_{i-1}([0,T])=\qbin{n-1}{i-1}{q}(-1)^{i-1}q^{\binom{i-1}{2}}.
\end{equation}
Moreover,
\begin{equation}\label{iiiii}
\left| \{U \in \mH(q,n,n-1) \mid \rk(U)=1, \ U \not\le T\}\right| = \sum_{j=1}^{n-1} \binom{n-1}{j-1} (q-1)^{j-1}.
\end{equation}
Combining (\ref{iii}), (\ref{iiii}) and (\ref{iiiii}) one concludes the proof.
\end{proof}

The following theorem gives explicit formul{\ae} for the Whitney numbers of $\mH(q,n,2)$. In Section~\ref{sec:HWDL5} we will combine these with Theorem \ref{thm:dual_appl} to obtain precise information on the asymptotics of $w_2(q,n,d)$ as $q \to +\infty$. The next result is essentially already present in Dowling's paper \cite{dowling1973class}.

\begin{theorem}[Dowling] \label{th:fw2}
For all $n \ge 2$, the Whitney numbers of $\mH(q,n,2)$ are given by the formul{\ae}
$$w_0(q,n,2)=1, \qquad (-1)^i  w_i(q,n,2)=  \sum_{1 \le j_1 < \cdots < j_i \le n} \ \prod_{t=1}^i \left(1+(j_t-1)(q-1) \right) \quad \mbox{for $i \ge 1$}.$$
\end{theorem}

\begin{proof}
By Theorem \ref{th:down2}, $\mH(q,n,2)$ is the Dowling lattice $Q_n(\F_q^*)$.
This lattice is supersolvable of rank~$n$, and its characteristic polynomial is known to be 
$$\prod_{i=1}^{n} \left( \lambda - (1+(i-1)(q-1))\right);$$
see e.g. \cite[Theorem 5]{dowling1973class}. 
Let $\gamma_i=1+(i-1)(q-1)$ for all $1 \le i \le n$. Then 
$$\chi(\mH(q,n,2);\lambda)= \prod_{i=1}^{n} (\lambda-\gamma_i)= \sum_{i=0}^n
(-1)^i \; \EL_i(\gamma_1,...,\gamma_n) \; \lambda^{n-i},$$
where $\EL_i$ denotes the $i$-th elementary symmetric function. Therefore the $i$-th Whitney number of $\mH(q,n,2)$ is given by
$w_i(q,n,2)= (-1)^i \, \EL_i(\gamma_1,...,\gamma_n)$.
\end{proof}

We also give an explicit expression for the Whitney numbers of 
the lattice $\mH(q,n,n-2)$. These are the coefficients of $\chi(\mH(q,n,n-2);\lambda)$, which do not seem immediate to compute from the formula in Theorem \ref{th:bonin2}. We will therefore take an indirect approach based on Theorem \ref{prop2}.

\begin{theorem} \label{th:n-2}
For all $n \ge 3$ and $2 \le i \le n$, the $i$-th Whitney number of $\mH(n,n-2)$ is given by 
\begin{multline*}
(-1)^i w_i(q,n,n-2)= \qbin{n}{i}{q}q^{\binom{i}{2}} - \left((q-1)^{n-1} + n(q-1)^{n-2}\right) \qbin{n-1}{i-1}{q}q^{\binom{i-1}{2}}  \\
+ \left((q-1)^{n-1} \prod_{j=2}^{n-2} (q-j)\right) \qbin{n-2}{i-2}{q}q^{\binom{i-2}{2}}.
\end{multline*}
\end{theorem}

\begin{proof}
It is easy to show that $\alpha_0(q,n,n-2)=1$ and $\alpha_1(q,n,n-2)=(q-1)^{n-1} + n(q-1)^{n-2}$. Moreover, from the proof of \cite[Theorem 3.5]{bonin1993modular} one sees that
$$\alpha_2(q,n,n-2)=(q-1)^{n-1} \prod_{j=2}^{n-2} (q-j).$$
Finally, for $k \ge 3$ we have $\alpha_k(q,n,n-2)=0$ by the Singleton Bound (Theorem \ref{th:sing}).
Therefore the desired formula follows from Theorem \ref{prop2}.
\end{proof}

We now present an application of Theorem \ref{th:compl}, giving a closed formula for the second Whitney number of the lattice $\mH(q,n,3)$, for all $n \ge 6$. As the following Example \ref{notc} shows, 
$\mH(q,n,3)$ has critical exponent strictly smaller than $n-2$ in general. 
In particular, this lattice has not been studied in previous work.

\begin{example} \label{notc}
For $n \ge 6$ and $q \ge n$, there exists a code $C \le \F_q^n$ with dimension $n-3$ and minimum Hamming distance $\dH(C)=4$. This follows, for example, from the Reed-Solomon construction~\cite[Chapter 10]{macwilliams1977theory}. Therefore, by the definition of critical exponent (Remark \ref{rem:imp}), we have
$\crit(\mH(q,n,3)) \le n-(n-3)=3<n-2$.
\end{example}

\begin{theorem} \label{th:comput3}
Let $n \ge 6$. We have 
\begin{equation*}
w_2(q,n,3) = (q-1)^4 \,
\sum_{\ell=1}^{n-5} \left( \binom{n-\ell}{2}\binom{n-\ell-2}{3} \right) +\sum_{t=2}^5 \binom{n}{t} w_2(q,t,3)\sum_{s=t}^5 \binom{n-t}{n-s} (-1)^{s-t},
\end{equation*}
where the values of the Whitney numbers $w_2(q,t,3)$ for $t=2,3,4,5$ can be found in Example~\ref{ex:primo}, Theorem~\ref{th:fw2}, and Theorem~\ref{th:n-2} respectively.
\end{theorem}

We can also give a more explicit formula for $w_2(q,n,3)$ and $n \ge 6$. This will be needed in later section when discussing the asymptotics of the Whitney numbers of higher-weight Dowling lattices.

\begin{theorem} \label{th:comput3bis}
Let $n \ge 6$. We have 
\begin{align*}
w_2(q,n,3) &= \frac{1}{72}q^4n^6 - \frac{1}{12}q^4n^5 + \frac{1}{18}q^4n^4 + \frac{1}{2}q^4n^3 - \frac{77}{72}q^4n^2 +   \frac{7}{12}q^4n - \frac{1}{18}q^3n^6  \\ &+ \frac{5}{12}q^3n^5 - \frac{49}{72}q^3n^4 - \frac{7}{6}q^3n^3 + \frac{269}{72}q^3n^2 - \frac{9}{4}q^3n + \frac{1}{12}q^2n^6 - \frac{3}{4}q^2n^5 \\ &+ 2q^2n^4 - \frac{7}{12}q^2n^3 - \frac{43}{12}q^2n^2 + \frac{17}{6}q^2n - \frac{1}{18}qn^6 + \frac{7}{12}qn^5 - \frac{157}{72}qn^4   \\ &+ \frac{19}{6}qn^3 - \frac{55}{72}qn^2 - \frac{3}{4}qn + \frac{1}{72}n^6 - \frac{1}{6}n^5 + \frac{29}{36}n^4 - \frac{23}{12}n^3 + \frac{157}{72}n^2 - \frac{11}{12}n.
\end{align*}
\end{theorem}

In the proofs of Theorems \ref{th:comput3} and \ref{th:comput3bis} we will need the following identity involving sums and products of binomial coefficients.
The proof is in Appendix \ref{app:A}.

\begin{lemma} \label{lem:numer}
Let $d \ge 1$ and $n \ge 2d$ be integers. We have
\begin{equation} \label{fo}
\sum_{1 \le \ell_1<\ell_2 \le n-d+1} \binom{n-\ell_1-d}{d-1}\binom{n-\ell_2}{d-1} = \sum_{\ell=1}^{n-2d+1} \binom{n-\ell}{d-1}\binom{n-\ell-d+1}{d}.
\end{equation}
\end{lemma}

\begin{proof}[Proof of Theorem \ref{th:comput3}]
Applying Theorem \ref{th:compl} with $i=2$ we obtain
\begin{multline*}
w_2(q,n,3) = (q-1)^4
\sum_{1 \le \ell_1 < \ell_2 \le n-2}  \left( \binom{n-\ell_1-3}{2} \binom{n-\ell_2}{2} \right)  \\ + \sum_{t=2}^5 \binom{n}{t} w_2(q,t,3)\sum_{s=t}^5 \binom{n-t}{n-s} (-1)^{s-t}.
\end{multline*}
By Lemma \ref{lem:numer} we have
$$\sum_{1 \le \ell_1<\ell_2 \le n-2} \binom{n-\ell_1-3}{2}\binom{n-\ell_2}{2} = \sum_{\ell=1}^{n-5} \binom{n-\ell}{2}\binom{n-\ell-2}{3},$$
and the theorem follows.
\end{proof}

We can now establish also Theorem \ref{th:comput3bis}.

\begin{proof}[Proof of Theorem \ref{th:comput3bis}]
By Theorem \ref{th:comput3} we have
\begin{equation}
w_2(q,n,3) = (q-1)^4 A(n) + B(q,n),
\end{equation}
where
$$A(n) \defi \sum_{\ell=1}^{n-5} \binom{n-\ell}{2}\binom{n-\ell-2}{3},
 \qquad B(q,n) \defi \sum_{t=2}^5 \binom{n}{t} w_2(q,t,3)\sum_{s=t}^5 \binom{n-t}{n-s} (-1)^{s-t}.$$
In the remainder of the proof we will derive explicit expressions for $A(n)$ and $B(q,n)$. The computations are long and technical, and some parts will be omitted.

We first compute $A(n)$. Observe that
\begin{eqnarray*}
f_{n,\ell} \defi  \binom{n-\ell}{2}\binom{n-\ell-2}{3} = \frac{(n-\ell)(n-\ell-1)}{2} \cdot \frac{(n-\ell-2)(n-\ell-3)(n-\ell-4)}{3!}.
\end{eqnarray*}
Expanding the products one obtains
$$12 \cdot f_{n,\ell}= \sum_{j=0}^5 P_j \; \ell^j,$$
where
$$\left\{ \begin{array}{l} 
P_0= n^5-10n^4+35n^3-50n^2+24n, \\
P_1= -5n^4+40n^3-105n^2+100n-24, \\
P_2= 10n^3-60n^2+105n-50, \\
P_3= -10n^2+40n-35, \\
P_4= 5n-10, \\
P_5=-1.
\end{array} 
\right.$$
Setting $m=n-5$ and using Faulhaber's Formulas \cite[page 106]{conway2012book} we then compute
\begin{multline*}
12 \cdot \sum_{\ell=1}^{n-5}f_{n,\ell} = mP_0 + \frac{1}{2} (m^2+m)P_1 +
\frac{1}{6} (2m^3+3m^2+m)P_2 + \frac{1}{4}(m^4+2m^3+m^2)P_3 \\ +
\frac{1}{30} (6m^5+15m^4+10m^3-m)P_4 +
\frac{1}{12} (2m^6+6m^5+5m^4-m^2)P_5.
\end{multline*}
Substituting $m=n-5$, dividing by 12, and making the computations one finds
\begin{equation} \label{exprA}
A(n)=\sum_{\ell=1}^{n-5}f_{n,\ell}=\frac{1}{72}n^6 - \frac{5}{24}n^5 + \frac{85}{72}n^4 - \frac{25}{8}n^3 + \frac{137}{36}n^2 - \frac{5}{3}n,
\end{equation}
which is a closed expression for $A(n)$. 

Using Example \ref{ex:primo}, Theorem \ref{th:n-1} and Theorem \ref{th:n-2} one finds the following formul{\ae} for the Whitney number $w_2(q,t,3)$:
\begin{equation}
w_2(q,t,3) = \left\{   \begin{array}{cl}
q & \mbox{ if $t=2$,} \\
q^3+q^2+q & \mbox{ if $t=3$,} \\
3q^4+q^3+2q^2-q+1 & \mbox{ if $t=4$},\\ 
30q^4-55q^3+60q^2-35q+10 & \mbox{ if $t=5$}.
   \end{array} \right.
\end{equation}
Finally, using Eq. (\ref{exprA}) and the definition of $B(q,n)$, after tedious computations one finds
\begin{align*}
(q-1)^4A(n)+B(q,n) &= \frac{1}{72}q^4n^6 - \frac{1}{12}q^4n^5 + \frac{1}{18}q^4n^4 + \frac{1}{2}q^4n^3 - \frac{77}{72}q^4n^2 +   \frac{7}{12}q^4n - \frac{1}{18}q^3n^6  \\ &+ \frac{5}{12}q^3n^5 - \frac{49}{72}q^3n^4 - \frac{7}{6}q^3n^3 + \frac{269}{72}q^3n^2 - \frac{9}{4}q^3n + \frac{1}{12}q^2n^6 - \frac{3}{4}q^2n^5 \\ &+ 2q^2n^4 - \frac{7}{12}q^2n^3 - \frac{43}{12}q^2n^2 + \frac{17}{6}q^2n - \frac{1}{18}qn^6 + \frac{7}{12}qn^5 - \frac{157}{72}qn^4   \\ &+ \frac{19}{6}qn^3 - \frac{55}{72}qn^2 - \frac{3}{4}qn + \frac{1}{72}n^6 - \frac{1}{6}n^5 + \frac{29}{36}n^4 - \frac{23}{12}n^3 + \frac{157}{72}n^2 - \frac{11}{12}n,
\end{align*}
which is the desired formula.
\end{proof}

We conclude this section by computing the third Whitney number of $\mH(2,n,3)$
for $n$ sufficiently large. Part of our proof relies on computational results, and can be easily extended to other (small) values of $q$ and $d$. 

\begin{theorem}
For all $n \ge 9$ we have
\begin{align*}
-w_3(2,n,3)  &=  \sum_{1 \le \ell_1< \ell_2 < \ell_3 \le n-2} \left(  \prod_{j=1}^3  \binom{n-\ell_j-9+3j)}{2} \right) 
+ 8 \binom{n}{3}  \sum_{s=3}^{8} \binom{n-3}{n-s} (-1)^{s-3} \\ 
& +106 \binom{n}{4} \sum_{s=4}^{8} \binom{n-4}{n-s} (-1)^{s-4}
+  820 \binom{n}{5} \sum_{s=5}^{8} \binom{n-5}{n-s} (-1)^{s-5} \\ 
&+ 4565 \binom{n}{6}  \sum_{s=6}^{8} \binom{n-6}{n-s} (-1)^{s-6} 
+19810 \binom{n}{8}  \sum_{s=7}^{8} \binom{n-7}{n-s} (-1)^{s-7} 
+ 70728\binom{n}{8}.
\end{align*}

\end{theorem}

\begin{proof}
Using Theorem \ref{th:compl} we find

\begin{equation*}
-w_3(2,n,3)  =  \sum_{1 \le \ell_1< \ell_2 < \ell_3 \le n-2} \left(  \prod_{j=1}^3  \binom{n-\ell_j-9+3j)}{2} \right) - 
 \sum_{t=3}^{8} \binom{n}{t} w_3(2,t,3)\sum_{s=t}^{8} \binom{n-t}{n-s} (-1)^{s-t}.
\end{equation*}

The values of $w_3(2,t,3)$ for $t \in \{3,...,8\}$ can be computed using a computer algebra software. These are the following:
$$-w_3(2,t,3)= \left\{  \begin{array}{cl}
8 & \mbox{if $t=3$,} \\
106 & \mbox{if $t=4$,} \\
820 & \mbox{if $t=5$,} \\
4565 & \mbox{if $t=6$,} \\
19810 & \mbox{if $t=7$,} \\
70728 & \mbox{if $t=8$.}
\end{array}
 \right.$$
 The desired theorem follows easily.
\end{proof}


 \section{HWDL -- Combinatorics of $w_2(q,n,d)$}
 \label{sec:HWDL4}
 
We entirely devote this section to the computation of the second Whitney number of $\mH(q,n,d)$, for all $q$, $n$ and $d$. We will express this number in terms of certain integers defined recursively, which we call ``agreement numbers''. 
As we will see, these can be given a precise combinatorial interpretation.
 
 \begin{definition} \label{def:cn}
 Let $a,b \ge 1$ and $c,\nu \ge 0$ be integers. The \textbf{agreement number} $\gamma_a(b,c,\nu)$ is defined via the recursion
 $$\gamma_a(b,c,\nu) \defi \sum_{s=0}^{c-1} \binom{b}{s} \gamma_{a-1}(b-s,c,\nu-s) + \sum_{s=c}^\nu \binom{b}{s}\binom{b-s}{\nu-s} 
 (a-2)^{\nu-s} \quad \mbox{for } \left\{ \begin{array}{l} a \ge 2, \\ b \ge c+1, \\ \nu \ge c,  \end{array}\right.$$
 {(where $0^0=1$)}
with initial conditions
$$\gamma_a(b,c,\nu) \defi \left\{ \begin{array}{cl} 
0 & \mbox{if $a=1$ or $b<c$ or $\nu < c$ or $\nu>b$,} \\ a-1 & \mbox{if $a \ge 2$ and $b=c=\nu$.} \end{array}\right.$$ 
 \end{definition}

 The main goal of this section is to establish the following formula, which expresses $w_2(q,n,d)$ in terms of the agreement numbers.
 
 \begin{theorem} \label{th:main2}
 For all $n \ge d \ge 2$, the second Whitney number of $\mH(q,n,d)$ is given by
 \begin{align*} w_2(q,n,d) &=(q^{n-1}-1) \, \sum_{j=1}^d \binom{n}{j} (q-1)^{j-2} \ - \sum_{1 \le \ell_1<\ell_2 \le n} \Biggl[ q^{n-\ell_1-1}\left( \sum_{j=0}^{d-1} \binom{n-\ell_2}{j} (q-1)^j\right)  \\
&+ \sum_{j=d}^{n-\ell_2}\sum_{h=0}^{d-1} \binom{n-\ell_2}{j}   \binom{n-\ell_1-1}{h}(q-1)^{j+h} \\ 
&+ \sum_{s=d}^{n-\ell_2} \ \sum_{t=0}^{d-2} \binom{n-\ell_2}{s} \binom{n-\ell_1-1-s}{t} (q-1)^{s+t} \sum_{\nu=d-t}^s \gamma_q(s,s-d+t+2,\nu) \Biggr].
 \end{align*} 
 \end{theorem}

  \begin{remark}
In Theorem \ref{th:polygamma}, we will show that $\gamma_a(b,c,\nu)$ is a polynomial in $a$, whose coefficients are expressions involving the Bernoulli numbers. We will then use this fact to prove that $w_2(q,n,d)$ is a polynomial in $q$. This does not seem immediate from Definition \ref{def:cn}
\end{remark}

 For the purpose of deriving Theorem \ref{th:main2}, it will be convenient to work with the following numerical quantities.
 
 \begin{notation} \label{notaz:beta}
 Let $q$ be a prime power. For $n \ge 1$ and $k,d \ge 0$, let
 $$\beta_k(q,n,d) \defi \qbin{n}{k}{q}-\alpha_k(q,n,d).$$
  \end{notation}

 By definition, $\beta_k(q,n,d)$ counts the number of $k$-dimensional codes $C \le \F_q^n$ with minimum Hamming distance $\dH(C) \le d$.
 Our plan is to use the following consequence of Theorem~\ref{prop2}.
 
 \begin{theorem} \label{th:beta}
 Let $n,d \ge 1$. For all $1 \le i \le n$ we have
$$(-1)^i w_i(q,n,d) = \sum_{k=1}^i \beta_k(q,n,d) \qbin{n-k}{i-k}{q} (-1)^{k-1} q^{\binom{i-k}{2}}.$$
 \end{theorem}
 
 \begin{proof}
 Fix an integer $1 \le i \le n$. 
By Theorem \ref{prop2} we have
\begin{equation}  \label{eest1}
w_i(q,n,d)= \sum_{k=0}^i \alpha_k(q,n,d) \qbin{n-k}{i-k}{q} (-1)^{i-k} q^{\binom{i-k}{2}}.
\end{equation}
By the definition of $\beta_k(q,n,d)$, Eq. (\ref{eest1}) can be re-written as
\begin{align}  
w_i(q,n,d) &= \sum_{k=0}^i \left( \qbin{n}{k}{q} -\beta_k(q,n,d) \right) \qbin{n-k}{i-k}{q} (-1)^{i-k} q^{\binom{i-k}{2}} \nonumber \\
&= \sum_{k=0}^i \qbin{n}{k}{q}\qbin{n-k}{i-k}{q} (-1)^{i-k} q^{\binom{i-k}{2}} \ - \  \sum_{k=1}^i \beta_k(q,n,d) \qbin{n-k}{i-k}{q} (-1)^{i-k} q^{\binom{i-k}{2}}.\label{eest2}
\end{align}
Using Lemma \ref{proprq} multiple times we compute
\begin{eqnarray*}
\sum_{k=0}^i \qbin{n}{k}{q}\qbin{n-k}{i-k}{q} (-1)^{i-k} q^{\binom{i-k}{2}} &=&
\sum_{k=0}^i \qbin{n}{n-k}{q}\qbin{n-k}{n-i}{q} (-1)^{i-k} q^{\binom{i-k}{2}} 
\\
&=&
\qbin{n}{n-i}{q}\ \sum_{k=0}^i \qbin{i}{k}{q}(-1)^{i-k} q^{\binom{i-k}{2}}  \\
&=&
\qbin{n}{n-i}{q}\ \sum_{k=0}^i \qbin{i}{i-k}{q}(-1)^{i-k} q^{\binom{i-k}{2}} \\
&=& \qbin{n}{n-i}{q}\ \sum_{k=0}^i \qbin{i}{k}{q}(-1)^{i} q^{\binom{i}{2}}=0,
\end{eqnarray*}
where the latter equality follows from the $q$-Binomial Theorem (note that $i \ge 1$ by assumption).
Therefore Eq.~(\ref{eest2}) simplifies to
\begin{equation*} \label{eest3}
w_i(q,n,d)=-\sum_{k=1}^i \beta_k(q,n,d) \qbin{n-k}{i-k}{q} (-1)^{i-k} q^{\binom{i-k}{2}}. \qedhere
\end{equation*}
 \end{proof}
 
 \begin{remark} \label{rem:look}
 In order to obtain Theorem \ref{th:main2} from Theorem \ref{th:beta}, we would now need an explicit expression for $\beta_1(q,n,d)$ and $\beta_2(q,n,d)$. By definition, we have
 $$\beta_1(q,n,d)= \sum_{j=1}^d \binom{n}{j} (q-1)^{j-1}\quad \mbox{for all $n \ge d \ge 1$.}$$
In the remainder of this section we compute $\beta_2(q,n,d)$, i.e., the number of $2$-dimensional codes $C \le \F_q^n$ of minimum Hamming distance $\le d$. We will give a closed formula for this quantity involving the agreement numbers,  introduced in Definition \ref{def:cn} via a recursion.
\end{remark}


 We start by providing a combinatorial interpretation for $\gamma_a(b,c,\nu)$, that also explains the choice of the terminology. 
 In the sequel, given a non-empty set $A$, an element $* \in A$, and an integer $b \ge 1$, we let
 $\wH^*(v)=|\{1 \le i \le b \mid v_i \neq *\}|$ for all $v \in A^b$.
 
 \begin{lemma} \label{lem:intc}
 Let $a,b \ge 1$ and $c,\nu \ge 0$ be integers. We have
$$
\gamma_a(b,c,\pe) = |\{v \in A^b \mid  \wH^*(v)= \pe, \ \exists \; S \subseteq \{1,...,b\} \mbox{ with } |S| \ge c \mbox{ and } v_i=v_j \neq *  \ \forall \  i,j \in S\}|,$$
where $A$ is any finite set of cardinality $a$, and $* \in A$ is a fixed element. In other words, $\gamma_a(b,c,\pe)$ counts the number of arrays $v \in A^b$ with the following properties:
\begin{itemize}[noitemsep]
\item $v$ has exactly $\nu$ entries different from $*$,
\item the entries of $v$ agree, and are not equal to $*$, in at least $c$ positions. 
\end{itemize}
  \end{lemma}
  
  \begin{proof} 
  For all $a,b \ge 1$ and $c,\nu \ge 0$ define the set
$$
\Gamma_a(b,c,\nu) \defi \{v \in A^b \mid  \wH^*(v)= \pe, \ \exists \; S \subseteq \{1,...,b\} \mbox{ with } |S| \ge c \mbox{ and } v_i=v_j \neq * \ \forall \ i,j \in S\},$$
where $A$ is any set of cardinality $a$, and $* \in A$ is a fixed element. It is easy to see that cardinality of the set $\Gamma_a(b,c,\nu)$, which we denote by $\gamma'_a(b,c,\nu)$ in the sequel, does not depend on the choice of~$A$ and of~$*$. Therefore in order to prove the lemma it suffices to show that the numbers in the set
$\{\gamma'_a(b,c,\nu) \mid a,b \ge 1 \mbox{ and } c,\nu \ge 0\}$ satisfy the recursion and the initial conditions of Definition~\ref{def:cn}. 

We have
$\Gamma_a(b,c,\nu)= \emptyset$ if $a=1$ or $b<c$ or $\nu<c$ or $\nu >b$. Moreover, if $a \ge 2$ and $b=c=\nu$, then $\Gamma_a(b,c,\nu)$ contains exactly $a-1$ arrays.  
Now suppose $a \ge 2$, $b \ge c+1$, and $\nu \ge c$. Fix a symbol $x \in A$ with $x \neq *$. For $S \subseteq \{1,...,b\}$ let
$$\Gamma_a^S(b,c,\nu) \defi \{v \in \Gamma_a(b,c,\nu) \mid S=\{1 \le i \le b \mid v_i =x\}\}.$$ Then
\begin{equation} \label{sss1}
\gamma'_a(b,c,\nu)=\left| \Gamma_a(b,c,\nu)\right| = \sum_{S \subseteq \{1,...,b\}} \left| \Gamma_a^S(b,c,\nu)\right|.
\end{equation}
It is not difficult to see that for all $S \subseteq \{1,...,b\}$ one has
\begin{equation} \label{sss2}\left| \Gamma_a^S(b,c,\nu)\right|= \left\{ \begin{array}{cl} \displaystyle \binom{b-|S|}{\nu-|S|}(a-2)^{\nu-|S|} & \mbox{ if $c \le |S| \le \nu$,} \\ \\  \gamma'_{a-1}(b-|S|,c,\nu-|S|) & \mbox{ if $|S| < c$,} \\ \\ 0 & \mbox{ otherwise.}\end{array}\right.
\end{equation}
The desired result can now be obtained combining (\ref{sss1}) and (\ref{sss2}).
\end{proof}

The next step towards the proof of Theorem \ref{th:main2} is to count the number of vectors $v \in \F_q^n$ such that $\wH(v)=\nu$ and $\xi v +w$ contains a vector of Hamming weight upper bounded by $d$ for some $\xi \in \F_q \setminus \{0\}$, where $w \in \F_q^n$ is a vector with full support.  This quantity can be expressed in terms of the agreement numbers via Lemma \ref{lem:intc}.

\begin{lemma} \label{lem:gvect}
Let $n,d \ge 1$ and $\nu \ge 0$ be integers, and let $w \in \F_q^n$ with $\sH(w)=\{1,...,n\}$. The number of $v \in \F_q^n$ with $\wH(v)=\nu$ and such that the set $\{\xi v + w \mid \xi \in \F_q \setminus \{0\}\}$ contains a vector of Hamming weight at most $d$ is $\gamma_q(n,n-d,\nu)$.
\end{lemma}
\begin{proof}
For $w \in \F_q^n$, let $R(w) = \{v \in \F_q^n \mid \wH(v)=\nu, \, \mid \exists \, \xi \in \F_q \setminus \{0\} \mbox{ with } \wH(\xi v +w) \le d\}$. Now fix $w \in \F_q^n$ with $\sH(w)=\{1,...,n\}$, and denote by $\bf 1$ the all-1 vector in $\F_q^n$. We claim that  $R(w)$ and $R({\bf 1})$ have the same cardinality. To see this, define
$\varphi:
R(w) \to R({\bf 1})$ by $(v_1,...,v_n) \mapsto (v_1/w_1,...,v_n/w_n)$. We will show that $\varphi$ is a well-defined bijection.

Suppose that $v \in R(w)$. Then there exists
$\xi \in \F_q \setminus \{0\}$ and a set $S \subseteq \{1,...,n\}$ with $|S| \ge n-d$ and $\xi v_i+w_i=0$ for all $i \in S$. Since $w$ has support $\{1,...,n\}$, the latter is equivalent to
$\xi v_i /w_i+1=0$ for all $i \in S$. Moreover, $(v_1,...,v_n)$ and $(v_1/w_1,...,v_n/w_n)$ have the same weight.
Thus $(v_1/w_1,...,v_n/w_n) \in R({\bf 1})$.
This shows that $\varphi$ is well-defined. Using the same argument one sees that the map $\psi:R({\bf 1}) \to R(w)$ given by $(v_1,...,v_n) \mapsto (v_1w_1,...,v_nw_n)$ is also well-defined. Since $\varphi$ and $\psi$ are the inverse of each other, we have that $\varphi$ is bijective, as claimed.

To conclude the proof, observe that a vector $v \in \F_q^n$ of weight $\nu$ belongs to $R({\bf 1})$ if and only if there exists $\xi \in \F_q^n \setminus \{0\}$ and a set $S \subseteq \{1,...,n\}$ with $|S| \ge n-d$ and $v_i =-1/\xi$ for all $i \in S$. The number of such vectors is the agreement number $\gamma_q(n,n-d,\nu)$ by Lemma \ref{lem:intc}.
\end{proof}

We are now ready to apply Lemmas \ref{lem:intc} and \ref{lem:gvect} to an enumerative combinatorics problem in coding theory. By Remark \ref{rem:look}, this will be the last step towards the proof of Theorem \ref{th:main2}.

\begin{theorem} \label{thm:beta2expl}
Let $n \ge d \ge 2$ be integers. The number of $2$-dimensional codes $C \le \F_q^n$ with minimum Hamming distance $\dH(C) \le d$ is
 \begin{align*}
\beta_2(q,n,d) &=  \sum_{1 \le \ell_1<\ell_2 \le n} \ \Biggl[ q^{n-\ell_1-1}\left( \sum_{j=0}^{d-1} \binom{n-\ell_2}{j} (q-1)^j\right)  \\ 
&+ \sum_{j=d}^{n-\ell_2}\sum_{h=0}^{d-1} \binom{n-\ell_2}{j}   \binom{n-\ell_1-1}{h}(q-1)^{j+h} \\ 
&+ \sum_{s=d}^{n-\ell_2} \ \sum_{t=0}^{d-2} \binom{n-\ell_2}{s} \binom{n-\ell_1-1-s}{t} (q-1)^{s+t} \sum_{\nu=d-t}^s \gamma_q(s,s-d+t+2,\nu) \Biggr].
\end{align*}
\end{theorem}

 \begin{proof}
 We count the number of $2 \times n$ matrices $G$ over $\F_q$ in reduced row-echelon form whose rows span a 2-dimensional code of minimum Hamming distance at most $d$.
Any $2 \times n$ matrix of rank 2 in reduced row-echelon form has the structure
 \begin{equation} \label{str}
 G= \begin{pmatrix} 
 0 & \cdots & 0 & 1 & \bullet & \cdots & \bullet & 0 & \bullet & \cdots & \bullet \\ 
 0 & \cdots & \cdots & \cdots & \cdots & \cdots & \cdots & 1 & \bullet & \cdots & \bullet
  \end{pmatrix},
  \end{equation}
  where the entries marked with a bullet are free. Given integers $1 \le \ell_1 < \ell_2 \le n$, we denote by $\mM(\ell_1,\ell_2)$ the set of matrices in the form of (\ref{str}) whose pivot columns are in positions $\ell_1$ and $\ell_2$, respectively. Moreover, we denote the $i$-th row of a matrix $G \in \mM(\ell_1,\ell_2)$ by $G_i$. Define
\begin{align*}
\mM_1(\ell_1,\ell_2) &= \{M \in \mM(\ell_1,\ell_2) \mid \wH(G_2) \le d\}, \\
\mM_2(\ell_1,\ell_2) &= \{M \in \mM(\ell_1,\ell_2) \mid \wH(G_2) >d, \ \wH(G_1) \le d \} \\
 \mM_3(\ell_1,\ell_2) &= \{M \in \mM(\ell_1,\ell_2) \mid \wH(G_2) > d, \ \wH(G_1)>d, \ \langle G_1,G_2 \rangle \mbox{ has min. H. distance $\le d$}\}.
\end{align*}
Since the three sets above are disjoint for any $\ell_1, \ell_2$, we have 
  \begin{equation} \label{eq:firstsum}
  \beta_2(q,n,d) = \sum_{1 \le \ell_1<\ell_2 \le n} |\mM_1(\ell_1,\ell_2)| + |\mM_2(\ell_1,\ell_2)| + |\mM_3(\ell_1,\ell_2)|.
  \end{equation}
 We will compute the cardinalities of $\mM_1(\ell_1,\ell_2)$, $\mM_2(\ell_1,\ell_2)$ and $\mM_3(\ell_1,\ell_2)$  for a given choice of the pivot indices $1 \le \ell_1 < \ell_2 \le n$. The size of $\mM_1(\ell_1,\ell_2)$ can be expressed as
 \begin{equation} \label{eq:firstterm}
 |\mM_1(\ell_1,\ell_2)| = q^{n-\ell_1-1} \cdot |\{v \in \F_q^{n-\ell_2} \mid \wH(v) \le d-1\}| = q^{n-\ell_1-1} \sum_{j=0}^{d-1} \binom{n-\ell_2}{j} (q-1)^j.
 \end{equation}
 The size of $\mM_2(\ell_1,\ell_2)$ is
 \begin{equation} \label{eq:midterm}
 |\mM_2(\ell_1,\ell_2)| =  \sum_{j=d}^{n-\ell_2} \binom{n-\ell_2}{j} (q-1)^j \sum_{h=0}^{d-1} \binom{n-\ell_1-1}{h}(q-1)^h.
 \end{equation}

 The computation of $|\mM_3(\ell_1,\ell_2)|$ is more involved (here is where the agreement numbers arise). Let $K=\{i \in \N \mid \ell_2 < i \le n\} \subseteq \{1,...,n\}$. For a subset $S \subseteq K$, define $$K(S)=\{i \in \N \mid \ell_1 < i < \ell_2\} \cup (K \setminus S).$$ Note that $|K|=n-\ell_2$ and 
 $|K(S)|=n-\ell_1-1-|S|$ for all $S \subseteq K$. We have
 \begin{align}
 |\mM_3(\ell_1,\ell_2)| &= \sum_{\substack{S \subseteq K \\ |S| \ge d}}
 |\{G \in \mM_2(\ell_1,\ell_2)\}\mid \sH(G_2) \cap K=S| \nonumber \\
 &= \sum_{\substack{S \subseteq K \\ |S| \ge d}} \ \sum_{T \subseteq K(S)}
 |\{G \in \mM_2(\ell_1,\ell_2) \mid \sH(G_2) \cap K=S, \ \sH(G_1) \cap K(S)=T\}|. \label{fdec}
 \end{align}
 Now fix subsets $S \subseteq K$ with $|S| \ge d$, $T \subseteq K(S)$, and a matrix $G \in \mM(\ell_1,\ell_2)$ with $\sH(G_2) \cap K=S$ and $\sH(G_1) \cap K(S)=T$. Denote by $s$ and $t$ the cardinalities of $S$ and $T$, respectively.
 If $t>d+2$, then $|\{G \in \mM_2(\ell_1,\ell_2) \mid \sH(G_2) \cap K=S, \ \sH(G_1) \cap K(S)=T\}|=0$.

 We henceforth assume $t \le d-2$.
 Observe that $G \in \mM_2(\ell_1,\ell_2)$ if and only if $\wH(G_1) \ge d+1$ and there exists $\xi \in \F_q \setminus \{0\}$ with $1 \le \wH(\xi G_1 + G_2) \le d$. Now observe that $\wH(G_1) = t+1+\wH(\pi_S(G_1))$. Moreover,
 for $\xi \in \F_q \setminus \{0\}$ one has
$$ \wH(\xi G_1+ G_2) = 2+t+\wH(\xi \pi_S(G_1) +\pi_S(G_2)).$$
  Therefore
  $\wH(G_1) \ge d+1$  and $1 \le \wH(\xi G_1 +G_2) \le d$ if and only if $\wH(\pi_S(G_1)) \ge d-t$ and there exists
  $\xi \in \F_q \setminus \{0\}$ with  
   $\wH(\xi \pi_S(G_1) + \pi_S(G_2)) \le d-t-2$. 
      To summarize,  
 $G \in \mM_2(\ell_1,\ell_2)$ if and only if $t \le d-2$, $\wH(\pi_S(G_1)) \ge d-t$, and there exists $\xi \in \F_q \setminus \{0\}$ with $\wH(\xi \pi_S(G_1) + \pi_S(G_2)) \le d-t-2$. By Lemma \ref{lem:gvect}, this shows that
 $$|\{G \in \mM_2(\ell_1,\ell_2) \mid \sH(G_2) \cap K=S, \ \sH(G_1) \cap K(S)=T\}| = (q-1)^{s+t} \sum_{\nu=d-t}^s \gamma_q(s,s-d+t+2,\nu).$$
 Finally, using (\ref{fdec}) we get
 \begin{equation*} \label{eq:secondterm}
 |\mM_2(\ell_1,\ell_2)| =\sum_{s=d}^{n-\ell_2} \ \sum_{t=d-s}^{d-2} \binom{n-\ell_2}{s} \binom{n-\ell_1-1-s}{t}
 (q-1)^{s+t}\sum_{\nu=d-t}^s  \; \gamma_q(s,s-d+t+2,\nu).
 \end{equation*}
 Combining this expression with (\ref{eq:firstsum}), (\ref{eq:firstterm}) and (\ref{eq:midterm}) one obtains the  desired formula.
 \end{proof}
 
The main result of this section, Theorem \ref{th:main2}, is now a straightforward consequence of Theorem~\ref{th:beta} and Theorem \ref{thm:beta2expl}, as described in Remark \ref{rem:look}.


\section{HWDL -- Density of Error-Correcting Codes}  \label{sec:HWDL5a}

 Although giving explicit formul{\ae} for the numbers $\beta_k(q,n,d)$ and $w_i(q,n,d)$ seems to be a very difficult task in general, quite detailed information can be obtained on their \textit{asymptotic behaviour} as the field size grows. In the next two sections we therefore study the asymptotics of $\beta_k(q,n,d)$ and $|w_i(q,n,d)|$ as $q \to +\infty$.

This section is devoted to the asymptotics of $\beta_k(q,n,d)$, which is closely related to the problem of computing the {density} of codes having minimum Hamming distance bounded from above.
We start by establishing the notation.

\begin{notation} \label{notaz:landau}
All the asymptotic estimates in the sequel are for $q \to +\infty$, unless otherwise stated. The other parameters ($n$, $d$, $k$, and $i$) are assumed to be fixed and therefore treated as constants. We shall use 
the standard Bachmann-Landau notation (``Big O'', ``Little O'', and ``$\sim$'') to express estimates for the growth rate of real-valued functions defined on an infinite subset $A \subseteq \N$. See \cite[Chapter 1]{de1981asymptotic} for a standard reference.
\end{notation}

The main result of this section is the following theorem, that provides a precise asymptotic estimate for the quantity $\beta_k(q,n,d)$.

\begin{theorem} \label{th:asbeta}
Fix integers $n > d \ge 2$. The following hold.
\begin{enumerate}[label=(\arabic*), noitemsep]
\item \label{pa1}
For all $1 \le k \le n-d$ we have
$$\beta_k(q,n,d) = \binom{n}{d} q^{(k-1)(n-k)+d-1} \; + o \left( q^{(k-1)(n-k)+d-1}\right).$$

\item \label{pa2}
In particular, 
$$\beta_k(q,n,d) \sim \left\{ \begin{array}{cl} 
\displaystyle 
\binom{n}{d} q^{(k-1)(n-k)+d-1} & \mbox{ if $1 \le k \le n-d,$} \\ 
q^{k(n-k)} & \mbox{ if $n-d+1 \le k \le n$.}\end{array} \right. $$

\item \label{pa3}
For all $1 \le k \le n$ we have
$$\beta_k(q,n,d) \in O \left( q^{(k-1)(n-k)+d-1}\right).$$
\end{enumerate}
\end{theorem}

We will establish Theorem \ref{th:asbeta} after a series of preliminary definitions and results. Applications of Theorem \ref{th:asbeta} to higher-weight Dowling lattices will be shown in Section~\ref{sec:HWDL5}.

 \begin{definition} \label{defpolyin}
 Let $A \subseteq \N$ be an infinite subset. We say that a function $f: A \to \R$ is a (\textbf{rational}) \textbf{polynomial} on $A$ if there exists a (unique) polynomial $p \in \Q[x]$ such that
 $f(a)=p(a)$ for all $a \in A$. We also say that $f$ is a \textbf{polynomial in $a$}, when the variable is clear from context.
 \end{definition}

Note that a polynomial $p$ as in Definition \ref{defpolyin} is necessarily unique by the Fundamental Theorem of Algebra and the infinitude of $A$.

The following preliminary result is well-known. It computes the growth rate of a polynomial, and states that the $q$-binomial coefficient is a polynomial in $q$.

\begin{lemma} \label{lem:aa}
\begin{enumerate}[label=(\arabic*), noitemsep]
\item \label{aa1} Let $r \ge 0$ be an integer and let $p=a_0+a_1x + \cdots + a_rx^r \in \Q[x]$ be a polynomial with $a_r>0$. We have $p(q) \sim a_rq^r$ as $q \to +\infty$, when $p$ is viewed as a function defined on an infinite subset of $\N$.
\item \label{aa2} For all integers $r \ge s \ge 0$ and all prime power $q$, the $q$-binomial coefficient of $r$ and $s$ is a polynomial in $q$ of degree $s(r-s)$ and with leading coefficient $1$. In particular, we have
$$\qbin{r}{s}{q} \sim q^{s(r-s)} \quad \mbox{as $q \to +\infty$}.$$
\end{enumerate}
\end{lemma}

The next lemma provides an upper bound for the number of $k$-dimensional codes
$C \le \F_q^n$ that have at least a non-zero codeword whose support is contained in a set $S_1 \subseteq \{1,...,n\}$, and at least a non-zero codeword whose support is contained in another set $S_2 \subseteq \{1,...,n\}$. The bound will be crucial in the proof of Theorem \ref{th:asbeta}. See Definition \ref{defavoids} for the notation.

\begin{lemma} \label{lem:2spc}
Let $n >d \ge 1$ be integers, and let $S_1,S_2$ be distinct $d$-subsets of 
$\{1,...,n\}$. For all $1 \le k \le n$ we have
$$|\{C \le \F_q^n \mid \rk(C)=k, \, C \mee \F_q^n(S_1), \, C \mee \F_q^n(S_2)\}| \le \frac{q^{d-1}-1}{q-1} \qbin{n-1}{k-1}{q} + \left(\frac{q^d-1}{q-1}\right)^2 \qbin{n-2}{k-2}{q}.$$
\end{lemma}

\begin{proof}
Let $A_1, A_2\subseteq \F_q^n$ be minimal sets of representatives for 
$\F_q^n(S_1)$ and $\F_q^n(S_2)$, respectively; see Definition \ref{def:msr}. Note that $|A_1|=|A_2|=(q^d-1)/(q-1)$. By Remark \ref{rmk:msr} we have
$$|\{C \le \F_q^n \mid \rk(C)=k, \, C \mee \F_q^n(S_1), \, C \mee \F_q^n(S_2)\}|  \ = \ 
|\{C \le \F_q^n \mid \rk(C)=k, \, C \mee A_1, \, C \mee A_2\}|.$$
Moreover, by definition,
\begin{equation} \label{bo1}
|\{C \le \F_q^n \mid \rk(C)=k, \, C \mee A_1, \, C \mee A_2\}| \le \beta+\gamma,
\end{equation}
where
\begin{eqnarray*}
\beta &=& |\{C \le \F_q^n \mid \rk(C)=k, \, C \mee A_1 \cap A_2\}|, \\
\gamma &=& |\{C \le \F_q^n \mid \rk(C)=k, \, C \mee A_1 \setminus A_2, \,
C \mee A_2 \setminus A_1\}|.
\end{eqnarray*}
The elements of $A_1 \cap A_2$ span a subspace of $\F_q^n(S_1) \cap \F_q^n(S_2)=\F_q^n(S_1 \cap S_2)$. Since $S_1$ and $S_2$ are distinct and both have cardinality $d$, the dimension of $\F_q^n(S_1 \cap S_2)$ is upper bounded by $d-1$. As a consequence, we have
$|A_1 \cap A_2| \le (q^{d-1}-1)/(q-1)$. Therefore we can apply Proposition \ref{prop:smalle} and conclude 
\begin{equation} \label{bo2}
\beta \le |A_1 \cap A_2| \cdot \qbin{n-1}{k-1}{q} \le \frac{q^{d-1}-1}{q-1} \qbin{n-1}{k-1}{q}.
\end{equation}
The next step is to obtain an upper bound for $\gamma$. It follows from the definitions that
\begin{equation*}
\gamma \le 
 \sum_{v \in A_1 \setminus A_2} \ \sum_{w \in A_2 \setminus A_1} |\{C \le \F_q^n \mid \rk(C)=k, \, \langle v, w \rangle \le C\}|.
\end{equation*}
For all $v \in A_1 \setminus A_2$ and $w \in A_2 \setminus A_1$ the space $\langle v,w \rangle$ has dimension two. Therefore
\begin{equation} \label{bo3}
\gamma \le |A_1 \setminus A_2| \cdot |A_2 \setminus A_1| \cdot \qbin{n-2}{k-2}{q} \le \left(\frac{q^d-1}{q-1}\right)^2 \qbin{n-2}{k-2}{q}.
\end{equation}
 The desired upper bound follows combining (\ref{bo1}), (\ref{bo2}) and (\ref{bo3}).
\end{proof}

We are now ready to establish the main results of this section.

\begin{proof}[Proof of Theorem \ref{th:asbeta}] \label{pth:asbeta}
The three statements are immediate if $k=1$, as 
$$\beta_1(q,n,k) = \sum_{j=1}^d \binom{n}{j} (q-1)^{j-1}.$$
We henceforth assume $k \ge 2$, and prove the statements separately.

Suppose $2 \le k \le n-d$, and denote by $S_1,...,S_L$ the $d$-subsets of $\{1,...,n\}$, where $L=\binom{n}{d}$. Note that $L>1$ since $n>d$ by assumption. We let $I=\{1,...,L\}$ to simplify the notation. For a subset $\emptyset \neq J \subseteq I$ define
$$\Gamma_J = \{C \le \F_q^n \mid \dim(C)=k, \, C \mee \F_q^n(S_\ell) \mbox{ for all $\ell \in J$}\}.$$
Then
$$\beta_k(q,n,d) = \left| \bigcup_{\emptyset \neq J \subseteq I} \Gamma_J \right|.$$
Using inclusion-exclusion we compute
\begin{equation} \label{psum}
\beta_k(q,n,d) = \sum_{\substack{J \subseteq I \\ |J|=1}} |\Gamma_J|  \ + \ \sum_{\substack{J \subseteq I \\ |J|\ge 2}} (-1)^{|J|+1} \ |\Gamma_J|.
\end{equation}
We treat separately the sets $J$ of cardinality one and those of cardinality two or more.
If $J \subseteq I$ has $|J|=1$, say $J=\{\ell\}$, then by Corollary \ref{coro:av} and the definition of $\Gamma_J$ we have
$$|\Gamma_J| = \qbin{n}{k}{q} - \alpha_k(\F_q^n,\F_q^n(S_\ell))= \sum_{i=1}^{\min\{k,d\}} (-1)^{i+1} q^{\binom{i}{2}} \qbin{d}{i}{q} \qbin{n-i}{k-i}{q},$$
as $\F_q^n(S_\ell) \le \F_q^n$ is a subspace of dimension $d$.
In particular, $|\Gamma_J|$ is a polynomial in $q$. To compute its leading coefficient observe that, by Lemma \ref{lem:aa}\ref{aa2}, for all $1 \le i \le \min\{k,d\}$ we have
$$(-1)^i q^{\binom{i}{2}} \qbin{d}{i}{q} \qbin{n-i}{k-i}{q} \sim q^{\binom{i}{2} + i(d-i)+(k-i)(n-k)}.$$ Using elementary methods from Calculus, one shows that the function $i \mapsto \binom{i}{2} + i(d-1)+(k-i)(n-k)$ attains its maximum for $i=1$ over the set $\{1,...,\min\{k,d\}\}$. Moreover, the value of such maximum is
$(k-1)(n-k)+d-1$. By Lemma \ref{lem:aa}\ref{aa1}, all of this shows that
\begin{equation} \label{ee1}
|\Gamma_J| = q^{(k-1)(n-k)+d-1} \ + \mbox{lower order terms in $q$,} \qquad \mbox{for all $J \subseteq I$ with $|J|=1$}.
\end{equation}

Now suppose that $J \subseteq I$ has cardinality $|J| \ge 2$. Take distinct elements $\ell_1,\ell_2 \in J$ and define $J'=\{\ell_1,\ell_2\} \subseteq J$. By definition, we have
$\Gamma_J \subseteq \Gamma_{J'}$. Thus by Lemma \ref{lem:2spc} we obtain
\begin{equation} \label{uu1}
|\Gamma_J| \le |\Gamma_{J'}| \le \frac{q^{d-1}-1}{q-1} \qbin{n-1}{k-1}{q} + \left(\frac{q^d-1}{q-1}\right)^2 \qbin{n-2}{k-2}{q}. 
\end{equation}
Using Lemma \ref{lem:aa}\ref{aa2} we compute the following asymptotic estimates:
$$\frac{q^{d-1}-1}{q-1} \qbin{n-1}{k-1}{q} \sim q^{(k-1)(n-k)+d-2}, \qquad 
\left(\frac{q^d-1}{q-1}\right)^2 \qbin{n-2}{k-2}{q} \sim q^{(k-2)(n-k)+2d-2}.$$
Since $k \le n-d$ by assumption, we have
$(k-2)(n-k)+2d-2 \le (k-1)(n-k)+d-2$. Therefore
\begin{equation} \label{uu2}
 \frac{q^{d-1}-1}{q-1} \qbin{n-1}{k-1}{q} + \left(\frac{q^d-1}{q-1}\right)^2 \qbin{n-2}{k-2}{q} \in O\left(q^{(k-1)(n-k)+d-2} \right). 
\end{equation}
Combining (\ref{uu1}) with (\ref{uu2}) we conclude
\begin{equation} \label{uu3}
|\Gamma_J| \in O\left(q^{(k-1)(n-k)+d-2} \right)   \qquad \mbox{for all $J \subseteq I$ with $|J| \ge 2$}.
\end{equation}
Finally, we obtain Part \ref{pa1} of the theorem by combining (\ref{psum}), (\ref{ee1}) and (\ref{uu3}).

Part \ref{pa2} follows from Part \ref{pa1} if $1 \le k \le n-d$. For $k \ge n-d+1$,
observe that by the Singleton Bound (Theorem \ref{th:sing}) we instead have
$$\beta_k(q,n,d)= \qbin{n}{k}{q} \sim q^{k(n-k)},$$
where the latter estimate follows from Lemma \ref{lem:aa}\ref{aa2}.
Finally, Part \ref{pa3} is a simple numerical consequence of Part \ref{pa2}.
\end{proof}

As a corollary of Theorem \ref{th:asbeta}, we can compute the asymptotics of the \textbf{density function} of the $k$-dimensional non-MDS codes in $\F_q^n$.
By definition, this is the map 
$$q \mapsto \delta_k(q,n) \defi \frac{\beta_k(q,n,n-k)}{\qbin{n}{k}{q}}.$$
The following result refines \cite[Corollary 5.2]{byrne2018partition} with new methods.

\begin{corollary}
Let $1 \le k \le n-2$ be integers. Then
$$\delta_k(q,n) \sim \binom{n}{k} q^{-1}.$$
\end{corollary}

\begin{proof}
Let $d=n-k \ge 2$. By Theorem \ref{th:asbeta} we have $\beta_k(q,n,d) \sim
\binom{n}{k} q^{(k-1)(n-k)+n-k-1}$. Therefore by Lemma \ref{lem:aa}\ref{aa2} we conclude
\begin{equation*}
\delta_k(q,n) \sim \frac{ \binom{n}{k} \, q^{(k-1)(n-k)+n-k-1}}{q^{k(n-k)}}=\binom{n}{k} q^{-1} \qedhere
\end{equation*}
\end{proof}

\section{HWDL -- Asymptotics of Whitney Numbers}  \label{sec:HWDL5}

We devote this section to the asymptotics of the Whitney numbers of higher-weight Dowling lattices as the field size grows. We follow the notation of the previous section; see Notation \ref{notaz:landau}.

\begin{remark}
It is well-known that the Whitney numbers (of the first kind) of a geometric lattice alternate in sign. In particular, we have $|w_i(q,n,d)|=(-1)^iw_i(q,n,d)$
for all $i,q,n,d$.
\end{remark}

We start by computing the exact growth rate of $|w_i(q,n,d)|$ for $d \in \{1,2,n-1,n\}$ and all $n$ and~$i$. These estimates can be derived from the explicit formul{\ae} obtained in Section \ref{sec:HWDL3}, with the aid of Lemma~\ref{lem:aa}\ref{aa2}.

\begin{theorem} \label{th:knownas}
Let $n \ge d \ge 1$ and $1 \le i \le n$ be integers. We have
\begin{equation}|w_i(q,n,d)| \sim \left\{ 
\begin{array}{ccl}
\displaystyle \binom{n}{i} & & \mbox{if $d=1$,} \\
\ \\ 
\left(\displaystyle \sum_{2 \le j_1 < \cdots < j_i \le n} \ \prod_{t=1}^i (j_t-1) \right) q^i & & \mbox{if $n \ge 2$, $d=2$, $1 \le i \le n-1$,} \\
\ \\ 
(n-1)! \; q^{n-1} & & \mbox{if $n \ge 2$, $d=2$, $i=n$,} \\
\ \\ 
q^{i(n-i) + \binom{i}{2}} & & \mbox{if $d \ge n$, $1 \le i \le n$,} \\
\ \\ 
(n-1) \; q^{in-1-\binom{i+1}{2}} & & \mbox{if $n \ge 3$, $d=n-1$, $2 \le i \le n$,} \\ 
\ \\ 
n \, q^{n-2} & & \mbox{if $n \ge 3$, $d=n-1$, $i=1$.} 
\end{array}
\right.
\end{equation}
\end{theorem}

\begin{proof}
The first formula follows from the fact that $\mH(q,n,1)$ is isomorphic to the Boolean algebra over the set $\{1,...,n\}$, as observed in Remark \ref{frems}. The second and third estimates can be obtained by direct inspection of Theorem \ref{th:fw2}. The fourth expression follows from the fact that $\mH(q,n,n)$ is the latice of subspaces of $\F_q^n$, whose Whitney numbers can be found in Example~\ref{ex:primo}. The fifth estimate is obtained from Theorem \ref{th:n-1} with the aid of Lemma \ref{lem:aa}\ref{aa2}, while the sixth follows from the fact that
\begin{equation*}
w_1(q,n,n-1) = - \sum_{j=1}^{n-1} \binom{n}{n-1} (q-1)^{j-1} \sim 
\binom{n}{n-1} q^{n-2} = n \, q^{n-2}. \qedhere
\end{equation*}
\end{proof}


\begin{remark}
In Section \ref{sec:HWDL3} we derived explicit formul{\ae} for $w_i(q,n,d)$ for values of $i,n,d$ that are not covered by Theorem \ref{th:knownas}. This is the case, for example, for the tuples $(i,n,d)$ of the form $(i,n,n-2)$ with $n \ge 3$ and those of the form 
$(2,n,3)$ with $n \ge 6$; see Theorems \ref{th:n-2} and \ref{th:comput3}. Obtaining asymptotic estimates for these Whitney numbers is not immediate.
We will determine their asymptotics later in Theorems \ref{th:exprexpl} and \ref{th:exprexpl2}, respectively.
\end{remark}

One of the main results of this section is the following upper bound for the growth rate of the absolute value of $w_i(q,n,d)$.

\begin{theorem} \label{th:as1}
Let $d \ge 2$ and $n \ge d+2$ be integers. For all $2 \le i \le n$ we have
$$|w_i(q,n,d)|\in O \left( q^{d-1+n(i-1) - \binom{i+1}{2}}\right).$$
\end{theorem}

The idea behind the proof is to extract information on the aymptotics of $|w_i(q,n,d)|$ from the asymptotics of the quantities $\beta_k(q,n,d)$ via Theorem~\ref{th:beta}.

\begin{proof}[Proof of Theorem \ref{th:as1}]
Fix $2 \le i \le n$. 
By Theorem \ref{th:beta} we have
\begin{equation} \label{abs}
|w_i(q,n,d)| =(-1)^i w_i(q,n,k)=  \sum_{k=1}^i \beta_k(q,n,d) \qbin{n-k}{i-k}{q} (-1)^{k-1} q^{\binom{i-k}{2}}.
\end{equation}
We will treat separately the summands corresponding to $k \in \{1,2\}$ and those corresponding to $k \ge 3$. We have
\begin{multline} \label{2summd}
\sum_{k=1}^2 \beta_k(q,n,d) \qbin{n-k}{i-k}{q} (-1)^{k-1} q^{\binom{i-k}{2}} = \left( \sum_{j=1}^d \binom{n}{j} (q-1)^{j-1} \right) \qbin{n-1}{i-1}{q}q^{\binom{i-1}{2}}  \\ + \left( \binom{n}{d} q^{n+d-3} + o \left( q^{n+d-3}\right)\right)\qbin{n-2}{i-2}{q} (-1) \; q^{\binom{i-2}{2}},
\end{multline}
where the latter estimate follows from Theorem \ref{th:asbeta}\ref{pa1}
and the fact that $n \ge d+2$ by assumption.
Applying Lemma \ref{lem:aa}\ref{aa2} one shows that
\begin{eqnarray*}
\left( \sum_{j=1}^d \binom{n}{j} (q-1)^{j-1} \right) \qbin{n-1}{i-1}{q}q^{\binom{i-1}{2}} &=& \binom{n}{d} q^{d+n(i-1)-\binom{i+1}{2}} + o \left(q^{d+n(i-1)-\binom{i+1}{2}} \right),\\
 \left( \binom{n}{d} q^{n+d-3} + o \left( q^{n+d-3}\right) \right)\qbin{n-2}{i-2}{q} (-1) \; q^{\binom{i-2}{2}}&=& \binom{n}{d} q^{d+n(i-1)-\binom{i+1}{2}} +o \left(q^{d+n(i-1)-\binom{i+1}{2}} \right).
\end{eqnarray*}
Therefore using~(\ref{2summd}) we conclude
\begin{equation} \label{eess1}
\sum_{k=1}^2 \beta_k(q,n,d) \qbin{n-k}{i-k}{q} (-1)^{k-1} q^{\binom{i-k}{2}} \in O \left( q^{d-1+n(i-1)-\binom{i+1}{2}}\right).
\end{equation}

We now turn to the summands in (\ref{abs}) corresponding to $k \ge 3$, deriving an asymptotic estimate for them. Note that these are only present when $i \ge 3$.
By Theorem \ref{th:asbeta}\ref{pa3} we have 
\begin{equation*} \label{eest4}
\beta_k(q,n,d) \in O \left( q^{(k-1)(n-k)+d-1}\right) \quad \mbox{for all $3 \le k \le n$}.
\end{equation*}
Therefore, for all $3 \le k \le i$, by Lemma \ref{lem:aa}\ref{aa2} we compute
\begin{equation} \label{eest5}
\left|\beta_k(q,n,d) \qbin{n-k}{i-k}{q} (-1)^{k-1} q^{\binom{i-k}{2}}\right| \in O \left( q^{\frac{1}{2} \left( -k^2+3k-i^2-i+2ni+2d-2-2n\right)}\right).
\end{equation}
The function $k \mapsto -k^2+3k-i^2-i+2ni+2d-2-2n$
attains its maximum for $k=3$ over the interval $\{3,...,i\}$. Moreover, the value of such maximum is
$$d-1+n(i-1) - \binom{i+1}{2}.$$
Combining this fact with Eq. (\ref{eest5}) we therefore obtain
\begin{equation} \label{eess2}
\sum_{k=3}^i \beta_k(q,n,d) \qbin{n-k}{i-k}{q} (-1)^{k-1} q^{\binom{i-k}{2}} \in O \left( q^{d-1+n(i-1)-\binom{i+1}{2}}\right).
\end{equation}
Since
\begin{eqnarray*}
|w_i(q,n,d)| &=&  \left| \sum_{k=1}^i \beta_k(q,n,d) \qbin{n-k}{i-k}{q} (-1)^{k-1} q^{\binom{i-k}{2}} \right| \\
&\le& \left| \sum_{k=1}^2 \beta_k(q,n,d) \qbin{n-k}{i-k}{q} (-1)^{k-1} q^{\binom{i-k}{2}} \right| + \left| \sum_{k=3}^i \beta_k(q,n,d) \qbin{n-k}{i-k}{q} (-1)^{k-1} q^{\binom{i-k}{2}} \right|,
\end{eqnarray*}
using (\ref{eess1}) and (\ref{eess2}) one obtains
\begin{equation*}
|w_i(q,n,d)| \in O \left( q^{d-1+n(i-1)-\binom{i+1}{2}}\right),
\end{equation*}
which is the desired result.
\end{proof}

It turns out that the asymtotic estimate of Theorem \ref{th:as1} is sharp for some values of the parameters. In the following theorem we compute the exact growth rate of $w_2(q,n,n-2)$ for all $n \ge 4$, and show that it meets the bound of Theorem \ref{th:as1}. The proof can be found in Appendix~\ref{app:A} and it uses the (partial) duality between the Whitney numbers of
$\mH(q,n,2)$ and $\mH(q,n,n-2)$, expressed by Theorem \ref{thm:dual_appl}.

\begin{theorem} \label{th:exprexpl}
For all $n \ge 4$ we have 
$$w_2(q,n,n-2) \sim \left( \frac{1}{8}n^4-\frac{3}{4}n^3+\frac{19}{8}n^2-\frac{11}{4}n \right) q^{2n-6}.$$
In particular, Theorem \ref{th:as1} is sharp for all $n \ge 4$ and $d=n-2$.
\end{theorem}

It is natural to ask if the estimate in Theorem \ref{th:as1} is sharp when $n$ is large with respect to $d$. The answer to this question is negative in general, as the following result shows.

\begin{corollary} \label{th:exprexpl2} Let $n \ge 6$. We have
\begin{equation} \label{eq:sta}
w_2(q,n,3) \sim \left(\frac{1}{72}n^6 - \frac{1}{12}n^5 + \frac{1}{18}n^4 - \frac{1}{2}n^3 + \frac{77}{72}n^2 - \frac{7}{12}n \right) q^4.
\end{equation}
\end{corollary}

\begin{proof}  The result follows from Theorem \ref{th:comput3bis} and the fact that the coefficient of $q^4$ in the right-hand side of (\ref{eq:sta}) is a polynomial in $n$ whose roots are $0$, $1$, $2$, $3$.
\end{proof}

In fact, when $n$ is large with respect to $d$ we can improve the bound of Theorem~\ref{th:as1} with the aid 
of Theorem \ref{th:compl}. The following result closes the section.

 \begin{theorem} \label{th:asnlarge}
Let $n,d \ge 3$ and $2 \le i \le n$ be integers. Suppose $n \ge id$. Then 
$$|w_i(q,n,d)| \in O \left( q^M \right), \quad \mbox{where } M= \max \left\{ i(d+1)-1-\binom{i+1}{2}, \ d(i^2-i+1)-i-\binom{i+1}{2} \right\}.$$
 \end{theorem}
 
 \begin{proof}
By Theorem \ref{th:compl}, we have
 \begin{multline} \label{ff}
w_i(q,n,d) \ = \ (-1)^i  (q-1)^{i(d-1)} \sum_{1 \le \ell_1< \ell_2 < \cdots < \ell_i \le n-d+1} \quad   \prod_{j=1}^i  \binom{n-\ell_j-d(i-j)}{d-1}
\\ + \sum_{t=i}^{id-1} \binom{n}{t} w_i(q,t,d)\sum_{s=t}^{id-1} \binom{n-t}{n-s} (-1)^{s-t}.
\end{multline}
We now derive upper bounds on the asymptotic growth of $w_i(q,t,d)$ for all integers $1 \le t \le id-1$, distinguishing four cases.

\underline{Case 1}: $i>t$. We have $w_i(q,t,d)=0$ by definition of Whitney number.

\underline{Case 2}: $i \le t \le d$. By Theorem \ref{th:knownas} we have $$w_i(q,t,d) \in O\left( q^{i(t-i) + \binom{i}{2}}\right).$$
Moreover, using the inequalities $2 \le i \le t \le d$ one easily shows that $$i(t-i) + \binom{i}{2} \le d(i^2-i+1)-i-\binom{i+1}{2}.$$

\underline{Case 3}: $i \le t =d+1$. Since $d \ge 3$ by assumption, we have $t \ge 4$. Thus again by Theorem \ref{th:knownas} we conclude
$$w_i(q,t,d)=w_i(q,d+1,d) \in O\left( q^{i(d+1)-1-\binom{i+1}{2}} \right).$$

\underline{Case 4}: $i \le t$ and $t \ge d+2$. Applying Theorem~\ref{th:as1} we obtain
$$w_i(q,t,d) \in O \left( q^{d-1+t(i-1) - \binom{i+1}{2}}\right).$$ Observing that
$$
d-1+t(i-1) - \binom{i+1}{2} = d-1+(id-1)(i-1) - \binom{i+1}{2} =
d(i^2-i+1)-i - \binom{i+1}{2},
$$ one gets 
$$w_1(q,t,d) \in O \left( q^{d(i^2-i+1)-i-\binom{i+1}{2}} \right).$$

Therefore, given the four previous analyses and the formula in (\ref{ff}), to conclude the proof it suffices to show that for all $i \ge 2$ and $d \ge 3$ we have
\begin{equation} \label{ff2}
i(d-1) \le d(i^2-i+1)-i-\binom{i+1}{2}.
\end{equation}
To see this, observe that for all $i \ge 2$ we have
$i^2-i+1 \ge i+ \frac{1}{3} \binom{i+1}{2}$,
as one verifies using elementary methods from Calculus. Therefore for $i \ge 2$ and $d \ge 3$ we have
$$d(i^2-i+1)-i-\binom{i+1}{2} \ge d \left(i+ \frac{1}{3} \binom{i+1}{2} \right)-i-\binom{i+1}{2} \ge id-i=i(d-1).$$
This establishes the inequality in (\ref{ff2}) and concludes the proof.
\end{proof}

\begin{remark}
Note that, by Theorem \ref{th:exprexpl2}, the bound of Theorem \ref{th:asnlarge} is sharp for $i=2$, $d=3$, and any $n \ge 6$.
\end{remark}

\section{HWDL -- Polynomiality in $q$ and Bernoulli Numbers}  \label{sec:HWDL6}

In this last section of the paper we discuss the polynomiality in $q$ of 
$\beta_k(q,n,d)$ and $w_i(q,n,d)$. These are clearly polynomials in $q$ for $k,i \in \{0,1\}$. We prove that they are polynomials also for
$k,i  =2$, and compute or upper bound their degrees. In the proof of these results, the celebrated Bernoulli numbers will arise.
We start by showing that $\gamma_a(b,c,\nu)$ is a polynomial in $a$. 


\begin{theorem} \label{th:polygamma}
Fix integers $b \ge 1$ and $c,\nu \ge 0$. The function $a \mapsto \gamma_a(b,c,\nu)$ defined on the set $\{a \in \N \mid a \ge 1\}$ is a polynomial in the variable $a$.
\end{theorem}
\begin{proof}
Observe that, by Definition \ref{def:cn}, $\gamma_a(b,c,\nu)$ is the evaluation at $a$ of the polynomial in $\Q[x]$
\begin{equation} \label{pini}
p^{b,c,\nu}= \left\{ \begin{array}{cl} 
x-1 & \mbox{if $b=c=\nu$,} \\ 0 & \mbox{if $b<c$ or $\nu<c$ or $\nu>b$}. \end{array} \right.
\end{equation}

We will show that for evey $b \ge 1$ and $c,\nu \ge 0$ with $\nu \ge c$ and
$b \ge c$
there exists
a polynomial $p^{b,c,\nu} \in \Q[x]$ such that $\gamma_a(b,c,\nu)=p^{b,c,\nu}(a)$ for all $a \ge 1$. We fix $c$ and proceed by induction on $b \ge c$. If $b=c$, then by (\ref{pini}) we have
that $\gamma_a(b,c,\nu)$ is the evaluation at $a$ of $x-1 \in \Q[x]$ if $\nu=c$, and of the zero polynomial otherwise. This establishes the induction hypothesis.

Now suppose that $b \ge c+1$, and that $\gamma_a(b',c,\nu')$ is a polynomial in $a$
for all $c \le b'<b$ and all $\nu' \ge c$.
We start by showing that 
\begin{equation} \label{eq:recu}
\gamma_a(b,c,\nu)=\sum_{i=2}^a \sum_{s=1}^{c-1} \binom{b}{s} \gamma_{i-1}(b-s,c,\nu-s)
+ \sum_{s=c}^\nu \binom{b}{s}\binom{b-s}{\nu-s} (i-2)^{\nu-s} \quad \mbox{for all $a \ge 1$},
\end{equation}
where the sum over an empty index set is zero by definition.
To see this, define  
\begin{equation} \label{def:fa}
f_a(b,c,\nu) = \sum_{s=1}^{c-1} \binom{b}{s} \gamma_{a-1}(b-s,c,\nu-s) + \sum_{s=c}^\nu \binom{b}{s}\binom{b-s}{\nu-s} 
 (a-2)^{\nu-s} \quad \mbox{for } a \ge 2.
\end{equation}
 Then
  \begin{equation} \label{chain}
 \gamma_a(b,c,\nu) = \sum_{i=2}^a f_i(b,c,\nu) \quad \mbox{ for all $a \ge 1$.}
 \end{equation}
 Indeed, using Definition \ref{def:cn} directly we find $\gamma_1(b,c,\nu)=0$ and, for all $a \ge 2$,
 $$\gamma_a(b,c,\nu) = \gamma_{a-1}(b,c,\nu)+f_a(b,c,\nu) =
\gamma_{a-2}(b,c,\nu) + f_{a-1}(b,c,\nu) +f_a(b,c,\nu) = \cdots 
= \sum_{i=2}^a f_i(b,c,\nu).$$
%
%
%
%
%
%
%
%
%

 We can now continue the inductive proof.
By the induction hypothesis and Eqs.~(\ref{pini}) and~(\ref{eq:recu})
one has
$$\gamma_a(b,c,\nu) = \sum_{i=2}^a r^{b,c,\nu}(i)$$
for a suitable polynomial $r^{b,c,\nu} \in \Q[x]$.
Write 
 $$r^{b,c,\nu}(x)=\sum_{j=0}^{t} r_j^{b,c,\nu} x^j,$$
 where $t=t^{b,c,\nu}$ denotes the degree of $r^{b,c,\nu}$. We then have
  \begin{equation} \label{chain2}
  \gamma_a(b,c,\nu) = \sum_{i=2}^a \, \sum_{j=0}^t r_j^{b,c,\nu} \, i^j
=\sum_{j=0}^t r_j^{b,c,\nu} \sum_{i=2}^a i^j \quad \mbox{ for all $a \ge 1$.}
\end{equation}
Using a celebrated formula of Faulhaber-Bernoulli-Jacobi \cite[page 106]{conway2012book} we can write
$$\sum_{i=2}^a i^j =-1+ \sum_{i=1}^a i^j =-1+\frac{1}{j+1} \sum_{i=0}^j (-1)^i \binom{j+1}{i} B_i \, a^{j+1-i},$$
where $B_i$ is the $i$-th Bernoulli number of the first kind, i.e.,
$$\frac{x}{e^x-1} = \sum_{i \ge 0} B_i \, \frac{x^i}{i!}.$$
Combining this with~(\ref{chain2}) we conclude that, for all $a \ge 1$, $\gamma_a(b,c, \nu)$ is the evaluation at $a$ of 
\begin{equation*}
p^{b,c,\nu}=\sum_{j=0}^t r_j^{b,c,\nu} \left(-1+ \frac{1}{j+1} \sum_{i=0}^j (-1)^i \binom{j+1}{i} B_i \, x^{j+1-i}\right) \in \Q[x].
\qedhere
\end{equation*}
\end{proof}

\begin{remark} \label{rem:ber}
The proof of Theorem \ref{th:polygamma} also shows how to recursively construct the polynomials~$p^{b,c,\nu} \in \Q[x]$ such that $\gamma_a(b,c,\nu)=p^{b,c,\nu}(a)$ for all
$a \ge 1$.
The recursion is given by
$$p^{b,c,\nu}=\sum_{j=0}^t u_j \left(-1+ \frac{1}{j+1} \sum_{i=0}^j (-1)^i \binom{j+1}{i} B_i \, x^{j+1-i}\right)  \quad \mbox{ for $b \ge c+1$ and $\nu \ge c$},$$
where
$$u=\sum_{s=1}^{c-1} \binom{b}{s} p^{b-s,c,\nu-s}(x-1) +\sum_{s=c}^\nu \binom{b}{s}\binom{b-s}{\nu-s} (x-2)^{\nu-s} \in \Q[x] \quad \mbox{ and } \quad t=\deg(u),$$
with initial condition
$$p^{b,c,\nu}= \left\{ \begin{array}{cl}  0 & \mbox{if $b<c$ or $\nu<c$ or $\nu>b$}, \\
x-1 & \mbox{if $b=c=\nu$.}  \end{array} \right.$$
\end{remark}

Combining Theorems \ref{th:polygamma}, \ref{th:main2} and \ref{thm:beta2expl} we obtain the following polynomiality result.

\begin{corollary}
For all $n \ge d \ge 2$, $\beta_2(q,n,d)$ and $w_2(q,n,d)$ are polynomials in $q$.
\end{corollary}

It is not clear if $\beta_k(q,n,d)$ or $w_i(q,n,d)$ are polynomials in $q$ for $k \ge 2$ and $i \ge 2$. If they are, their degrees can be computed or upper bounded as follows using Theorems \ref{th:asbeta}, \ref{th:as1}, and \ref{th:asnlarge}.

\begin{corollary}
Let $n>d \ge 2$ and $1 \le k \le n-d$ be integers. If $\beta_k(q,n,d)$ is a polynomial in~$q$, then its degree is $(k-1)(n-k)+d-1$.
\end{corollary}

\begin{corollary}
Let $d \ge 2$, $n \ge d+2$ and $2 \le i \le n$ be integers. If 
$w_i(q,n,d)$ is a polynomial in~$q$, then its degree is upper bounded by
$d-1+n(i-1)-\binom{i+1}{2}$. Moreover, if $d \ge 3$, $i \ge 2$, and 
 $n \ge id$, then its degree is upper bounded by
$$\max \left\{ i(d+1)-1-\binom{i+1}{2}, \ d(i^2-i+1)-i-\binom{i+1}{2} \right\}.$$
\end{corollary}

Note moreover that Corollary \ref{coro:relations} implies the following.

\begin{corollary}
Let $n,d \ge 1$ and $0 \le t \le n$ be integers. The following are equivalent:
\begin{enumerate} [label=(\arabic*), noitemsep]
\item $\alpha_k(q,n,d)$ is a polynomial in $q$ for all $0 \le k \le t$;
\item $w_i(q,n,d)$ is a polynomial in $q$ for all $0 \le i \le t$;
\item $\beta_k(q,n,d)$ is a polynomial in $q$ for all $0 \le k \le t$.
\end{enumerate}
\end{corollary}

%
%

\bigskip

\appendix

\section{Some Proofs} \label{app:A}

\begin{proof}[Proof of Theorem \ref{teo:av}] 
%
%
Fix $x \in \mL$ and define the integer
$$\Sigma \ \defi \ \sum_{\substack{t \in \mL \\ t \le x}} \ \sum_{\substack{x_A \in \mL(A), \: x_B \in \mL(B) \\ x_B^{A} =0\\ x_A \vee x_B\le t}} \: \mu_\mL(t,x) \:  \mu_{A}(x_A) \:  \mu_{B}(x_B).$$
On the one hand,
\begin{align}
\Sigma &= \sum_{\substack{x_A \in \mL(A), \: x_B \in \mL(B) \\ x_B^{A} =0}} \mu_{A}(x_A) \:  \mu_{B}(x_B) 
\sum_{\substack{t \in \mL \\ x_A \vee x_B \le t \le x}} \mu_\mL(t,x) \nonumber \\ &=
\sum_{\substack{x_A \in \mL(A), \: x_B \in \mL(B) \\ x_B^{A} =0 \\ x_A \vee x_B =x}} \mu_{A}(x_A) \: \mu_{B}(x_B).
\label{eq:11}
\end{align}
On the other hand, for all $t \in \mL$ we have
$$\Sigma_t \ \defi \ \sum_{\substack{x_A \in \mL(A), \: x_B \in \mL(A) \\ x_B^{A} =0 \\ x_A \vee x_B \le t}}  
\mu_{A}(x_A) \: \mu_{B}(x_B) = \sum_{\substack{x_B \in \mL(B) \\ x_B^{A} =0 \\ x_B \le t}} \mu_{B}(x_B)
 \sum_{\substack{x_A \in \mL(A) \\ x_A \le t}} \mu_{A}(x_A).$$
 By Lemma \ref{lem:toA} and the properties of $\mu_A$ we have 
 $$\sum_{\substack{x_A \in \mL(A) \\ x_A \le t}} \mu_{A}(x_A)=
\sum_{\substack{x_A \in \mL(A) \\ x_A \le t^{A}}} \mu_{A}(x_A)=0$$
 unless $t^{A}=0$, in which case 
 \begin{equation} \label{interme}
 \Sigma_t =  \sum_{\substack{x_B \in \mL(B) \\ x_B^{A} =0 \\ x_B \le t}} \mu_{B}(x_B) =  \sum_{\substack{x_B \in \mL(B) \\ x_B \le t}} \mu_{B}(x_B).
 \end{equation}
The latter equality follows from the fact that $x_B^{A} =0$ is implied by $x_B \le t$ and $t^{A}=0$.
Applying again Lemma \ref{lem:toA} to Equation (\ref{interme}) we conclude that $\Sigma_t=0$ unless $t^{A} = 0$ and $t^{B} = 0$,
in which case $\Sigma_t=1$. Since $\mL=\mL(A \cup B)$, for all $t \in \mL$ we have 
$t^{A} = 0$ and $t^{B}=0$ if and only if $t=0$.
Therefore using the definition of $\Sigma$ we obtain
\begin{equation}
\Sigma= \sum_{\substack{t \in \mL \\ t \le x}} \mu_\mL(t,x) \: \Sigma_t = 
\mu_\mL(x). \label{eq:22}
 \end{equation}
Now compare (\ref{eq:11}) and (\ref{eq:22}).
\end{proof}

\bigskip

\begin{proof}[Proof of Theorem \ref{twospacesS}]
By Corollary \ref{coro:three} we have
\begin{equation} \label{eeee} 
w_i(\mL)= \sum_{j=0}^i w_j([0,A_1]) \sum_{\substack{U \in \mL(A_2) \\ \rk(U)=i-j 
\\ U \cap A_1=\{0\}}} \mu_{A_2}(U).
\end{equation}
Note that $[0,A_1]$ is simply the lattice of subspaces of $A_1$. Moreover, an element $U \in \mL(A_2)$ satisfies $U \cap A_1=\{0\}$ if and only if
$U \cap (A_1 \cap A_2)=\{0\}$. Therefore
\begin{eqnarray*}
w_i(\mL) &=& \sum_{j=0}^i (-1)^jq^{\binom{j}{2}} \qbin{\rk(A_1)}{j}{q} \alpha_{i-j}(A_2,A_1 \cap A_2) \; (-1)^{i-j} q^{\binom{i-j}{2}} \\
&=& \sum_{j=0}^i (-1)^i q^{\binom{j}{2} + \binom{i-j}{2}} \qbin{\rk(A_1)}{j}{q} \ 
\sum_{h=0}^{i-j} (-1)^h q^{\binom{h}{2}} \qbin{\rk(A_1 \cap A_2)}{h}{q} \qbin{\rk(A_2)-h}{i-j-h}{q},
\end{eqnarray*}
where the latter equality follows from Corollary \ref{coro:av}. 
\end{proof}

\bigskip

\begin{proof}[Proof of Lemma \ref{lem:numer}]
There is a simple combinatorial way to obtain the desired identity: Both sides of Eq. (\ref{fo}) count the elements of the set
\begin{equation*}
S=\{(S_1,S_2) \mid S_1,S_2 \subseteq \{1,...,n\}, \, |S_1|=|S_2|=d, \, S_1 \cap S_2=\emptyset, \, \min(S_1)<\min(S_2)\}. 
\end{equation*}
Indeed, on the one hand we have
\begin{align*}
 |S|  & = \begin{multlined}[t]
  \sum_{1 \le \ell_1<\ell_2 \le n-d+1} |\{(S_1,S_2) \mid S_1,S_2 \subseteq \{1,...,n\}, \, |S_1|=|S_2|=d, \, S_1 \cap S_2=\emptyset, \\
  \min(S_1)=\ell_1, \, \min(S_2)=\ell_2\}|
 \end{multlined} \\
 & = \sum_{1 \le \ell_1<\ell_2 \le n-d+1} \binom{n-\ell_1-d}{d-1}\binom{n-\ell_2}{d-1}.
\end{align*}
On the other hand,
\begin{align*}
 |S|  & =
  \sum_{\ell=1}^{n-2d+1} |\{(S_1,S_2) \mid S_1,S_2 \subseteq \{1,...,n\}, \, |S_1|=|S_2|=d, \, S_1 \cap S_2=\emptyset,  \, 
  \min(S_1)=\ell \}| \\
 & = \sum_{\ell=1}^{n-2d+1} \binom{n-\ell}{d-1}\binom{n-\ell-d+1}{d}. \qedhere
\end{align*}
\end{proof}

\bigskip

\begin{proof}[Proof of Theorem \ref{th:exprexpl}]
Applying Theorem \ref{thm:dual_appl} with $d=2$ we obtain
$$\sum_{i=0}^{n-2} \qbin{n-i}{2}{2} w_i(n,2) = \qbin{n}{2}{q} + \qbin{n-1}{1}{q} w_1(n,n-2) +w_2(n,n-2).$$
Since
$$w_1(n,n-2)= -\sum_{j=1}^{n-2} \binom{n}{j}(q-1)^{j-1},$$
we have
$$w_2(n,n-2)= (q^{n-1}-1) \sum_{j=1}^{n-2} \binom{n}{j}(q-1)^{j-2} + \sum_{i=1}^{n-2} \qbin{n-i}{2}{q} w_i(n,2).$$
Thus by Theorem \ref{th:fw2} we conclude
\begin{multline} \label{deli1}
w_2(q,n,n-2)=(q^{n-1}-1) \sum_{j=1}^{n-2} \binom{n}{j} (q-1)^{j-2}  \\ + \ \sum_{i=1}^{n-2}  (-1)^i \qbin{n-i}{2}{q} \ \sum_{1 \le j_1 < \cdots < j_i \le n} \ \prod_{t=1}^i \left(1+(j_t-1)(q-1) \right).
\end{multline}
In particular, $w_2(q,n,n-2)$ is a polynomial in $q$. We will compute its degree and leading term. We start by observing that
\begin{align} \label{deli2}
  (q^{n-1}-1) &\sum_{j=1}^{n-2} \binom{n}{j} (q-1)^{j-2} 
 \notag \\
    &= (q^{n-1}-1) \left[ \binom{n}{n-2} (q-1)^{n-4} + \binom{n}{n-3} (q-1)^{n-5} + o\left(q^{n-5}\right) \right] \notag \\
    &= q^{n-1} \left[ \binom{n}{n-2}  \left( q^{n-4} - (n-4) q^{n-5} + o \left(q^{n-5} \right) \right) + \binom{n}{n-3} q^{n-5} + o \left(q^{n-5} \right)\right] \notag \\
    &= \binom{n}{n-2} q^{2n-5} + \left[ \binom{n}{n-3} - \binom{n}{n-2}(n-4)\right] q^{2n-6} + o\left(q^{2n-6}\right).
\end{align}
Using Lemma \ref{lem:aa}\ref{aa2}, one shows that for $i \ge 3$ we have
$$\qbin{n-i}{2}{q} \ \sum_{1 \le j_1 < \cdots < j_i \le n} \ \prod_{t=1}^i \left(1+(j_t-1)(q-1) \right) \in o\left(q^{2n-6} \right).$$ Therefore
\begin{align} \label{deli3}
\sum_{i=1}^{n-2}  (-1)^i \qbin{n-i}{2}{q} \ &\sum_{1 \le j_1 < \cdots < j_i \le n} \ \prod_{t=1}^i \left(1+(j_t-1)(q-1) \right) \notag \\
&= -\qbin{n-1}{2}{q} \left(n+ \binom{n}{2} (q-1) \right) \notag \\ & \qquad \qquad + \qbin{n-2}{2}{q} \ 
\sum_{0 \le j_1 < j_2 \le n-1} \ \prod_{t=1}^2 \left(1+j_t(q-1) \right) \ + o\left( q^{2n-6}\right).
\end{align}
Tedious computations show that
\begin{equation} \label{deli4}
\qbin{n-1}{2}{q} \left(n+ \binom{n}{2} (q-1) \right) =
 \binom{n}{n-2} q^{2n-5} + nq^{2n-6} + o\left( q^{2n-6}\right)
\end{equation}
and
\begin{equation} \label{deli5}
\qbin{n-2}{2}{q} \ 
\sum_{0 \le j_1 < j_2 \le n-1} \ \prod_{t=1}^2 \left(1+j_t(q-1) \right) = \left( \sum_{1 \le j_1 < j_2 \le n-1} j_1j_2\right) q^{2n-6} \ + o\left( q^{2n-6}\right).
\end{equation}
Now observe that
\begin{eqnarray} \label{deli6}
\sum_{1 \le j_1 < j_2 \le n-1} j_1j_2 &=& \sum_{i=2}^{n-1} i (1+ \cdots + i-1) \notag \\
&=& \sum_{i=2}^{n-1} i \ \frac{i(i-1)}{2} \ = \ \sum_{i=1}^{n-1} i \ \frac{i(i-1)}{2} \notag \\
&=& \  \sum_{i=1}^{n-1} i^3/2 \; - \; \sum_{i=1}^{n-1} i^2/2 \notag \\
&=& \frac{n^2(n-1)^2}{8} \; - \; \frac{n(n-1)(2n-1)}{12}. \label{deli4}
\end{eqnarray}
Therefore combining Eqs.~(\ref{deli3})--(\ref{deli6}) we  conclude
\begin{multline} \label{deli3a}
\sum_{i=1}^{n-2}  (-1)^i \qbin{n-i}{2}{q} \ \sum_{1 \le j_1 < \cdots < j_i \le n} \ \prod_{t=1}^i \left(1+(j_t-1)(q-1) \right)  \\
=- \binom{n}{n-2} q^{2n-5} \; - \; \left( n - \frac{n^2(n-1)^2}{8} + \frac{n(n-1)(2n-1)}{12}\right)  q^{2n-6} \ + o\left( q^{2n-6}\right).
\end{multline}
Using (\ref{deli1}), (\ref{deli2}) and (\ref{deli3a}) we then find
\begin{multline*}
w_2(q,n,n-2) \\ = \left( \binom{n}{n-3} - \binom{n}{n-2} (n-4) - n + \frac{n^2(n-1)^2}{8} - \frac{n(n-1)(2n-1)}{12} \right) q^{2n-6} \ +o\left( q^{2n-6}\right).
\end{multline*}
Lengthy computations show that
$$\binom{n}{n-3} - \binom{n}{n-2} (n-4) - n + \frac{n^2(n-1)^2}{8} - \frac{n(n-1)(2n-1)}{12} \ = \ \frac{1}{8}n^4-\frac{3}{4}n^3+\frac{19}{8}n^2-\frac{11}{4}n.$$
Since the latter expression is a polynomial in $n$ whose roots are $0$ and $2$, the desired asymptotic estimate for $w_2(q,n,n-2)$ follows for $n \ge 4$. 
\end{proof}

\bigskip


\bibliographystyle{siam}

\bibliography{bibliog2}

\end{document}